\documentclass[11pt,a4paper]{article}

\usepackage[T1]{fontenc}
\usepackage[utf8]{inputenc}
\usepackage{amsmath,amssymb,amsthm,amsopn}
\usepackage{mathtools}
\usepackage{bm}
\usepackage{calrsfs}
\usepackage{graphicx}
\usepackage[title]{appendix}
\usepackage{float}
\usepackage{placeins}
\usepackage{rotating}
\usepackage{comment}
\usepackage{enumitem}
\usepackage{multicol}
\usepackage{subcaption}
\usepackage{xcolor}
\usepackage[noadjust]{cite}
\usepackage[a4paper,margin=1in]{geometry}
\usepackage[unicode,colorlinks=true,
            linkcolor=blue,
            citecolor=blue,
            urlcolor=blue]{hyperref}

\newtheorem{theorem}{Theorem}[section]

\newtheorem{lemma}[theorem]{Lemma}

\theoremstyle{definition}
\newtheorem{definition}[theorem]{Definition}
\newtheorem{remark}[theorem]{Remark}
\newtheorem{example}[theorem]{Example}

\numberwithin{equation}{section}
\numberwithin{table}{section}
\numberwithin{figure}{section}

\newcommand{\doi}[1]{\url{https://doi.org/#1}}
\newcommand{\MR}[1]{\href{https://www.ams.org/mathscinet-getitem?mr=#1}{MR#1}}

\newcommand{\R}{\mathbb R}

\newcommand{\N}{\mathbb{N}}

\title{Periodic Solutions and Exponential Attractivity in Hopfield-Type Neural Networks with Generalized Piecewise Constant Arguments: Impulsive and Nonimpulsive Cases}

\author{
Ricardo Torres Naranjo%
\thanks{Corresponding author. Email: ricardo.torres@uach.cl}\\
\small Instituto de Ciencias F\'isicas y Matem\'aticas,
Facultad de Ciencias\\
\small Universidad Austral de Chile,
Campus Isla Teja, Valdivia, Chile
\\[1.5em]
\normalsize Samuel Castillo Apolonio
\thanks{ Email: scastill@ubiobio.cl}\\
\small Grupo de Investigaci\'on en Sistemas Din\'amicos
y Aplicaciones (GISDA)\\
\small Departamento de Matem\'atica,
Facultad de Ciencias\\
\small Universidad del B\'{\i}o-B\'{\i}o,
Casilla 5-C, Concepci\'on, Chile
}
\date{}

\begin{document}
\setlength{\parindent}{0pt}
\maketitle

\begin{center}
\small\itshape
Dedicated to the memory of Prof.\ Manuel Pinto J.
\end{center}

\begin{abstract}
\noindent
We investigate an impulsive Hopfield-type neural
network with a generalized piecewise constant
argument, including the nonimpulsive case in which
all jump terms vanish.
The piecewise constant argument is given by
$\gamma(t)=\zeta_k$ on each interval
$[t_k,t_{k+1})$, where
$t_k\leq\zeta_k\leq t_{k+1}$,
$k\in\mathbb{N}_0$.
Depending on the position of $\zeta_k$,
the argument may be advanced, retarded, or
of mixed type. By representing periodic solutions through a Green
function and applying a fixed-point argument,
we obtain explicit and verifiable sufficient
conditions for the existence and uniqueness
of an $\omega$-periodic solution.
We then use a Gronwall-type estimate, together
with period averages, to prove that this solution
is globally exponentially attractive.
These conditions can be checked directly from
the system parameters, without requiring
explicit knowledge of the periodic solution.
In the course of the analysis, we also revisit
and correct some estimates presented by
Pinto et al. \cite{Pinto_ejqtde}.
The case $\zeta_k=t_k$ reduces to their setting,
whereas our results cover the more general
choice of $\zeta_k$ within each interval.
We also present examples illustrating different
types of piecewise constant arguments.
\end{abstract}



\bigskip

\section{Introduction}
\label{intro}
\subsection{Scope}

Differential equations with piecewise constant arguments
have been studied for several decades. One of the
early observations in this direction was made by
Myshkis \cite{MYSH1}, who pointed out that a general
theory was still needed for differential equations
involving a discontinuous argument $r(t)$, such as
\[
x'(t)=f(t,x(t),x(r(t))).
\]
A simple example is the greatest integer function
\[
r(t)=[t]=n,\qquad
t\in[n,n+1),\quad n\in\mathbb Z.
\]
This function is constant on each interval
$[n,n+1)$ and has a discontinuity at every integer.
Differential equations involving arguments of this
kind are known as \emph{Differential Equations with
Piecewise Constant Arguments} (DEPCA).

The study of DEPCA began to develop systematically
in the 1980s, particularly through the work of
Shah and Wiener \cite{SHAH}. Around that time,
Busenberg and Cooke \cite{COOKE1} considered a
mathematical model for vertically transmitted
diseases in which the differential equation could
be reduced to a discrete equation. This was one
of the early examples showing how useful these
equations could be in applications. Over the years,
DEPCA have also appeared in problems arising in
biology, medicine, chemistry, physics, population
dynamics, and mechanical engineering; see
\cite{DAI,KARTAL,BOZKURT2}. For a detailed
presentation of the classical theory, we refer
to the book by Wiener \cite{WIENER1}.

A further development was introduced by Akhmet
\cite{AKH2}, who considered a more general class
of piecewise constant arguments. In this case,
the equation takes the form
\[
x'(t)=f(t,x(t),x(\gamma(t))),
\]
where $\gamma$ is defined using two sequences
$(t_k)_{k\in\mathbb Z}$ and
$(\zeta_k)_{k\in\mathbb Z}$ satisfying
\[
t_k<t_{k+1},\qquad
\lim_{k\to\pm\infty}t_k=\pm\infty,
\qquad
t_k\leq\zeta_k\leq t_{k+1}.
\]
More precisely, if $t\in I_k=[t_k,t_{k+1})$, then
\[
\gamma(t)=\zeta_k.
\]
These equations are called
\emph{Differential Equations with Piecewise Constant
Argument of Generalized Type} (DEPCAG).

Unlike the greatest integer function, the
generalized argument allows $\zeta_k$ to be
chosen anywhere between $t_k$ and $t_{k+1}$.
As a result, the equation may involve both
advanced and delayed terms.
Although $\gamma$ is discontinuous, solutions
of the corresponding nonimpulsive equations
remain continuous. Moreover, their values at
the points $t_k$ are related through a discrete
equation. In this way, continuous and discrete
dynamics appear naturally in the same problem.

Several properties of equations with piecewise
constant arguments, as well as related impulsive
and discrete systems, have been investigated.
These include convergence of solutions
\cite{BEREK1}, periodic solutions \cite{BEREK6},
stability \cite{AKH2}, asymptotic formulae
\cite{CTP2019}, controllability and observability
\cite{PINTO6}, and almost periodic solutions
of discrete-time neural networks \cite{HUANG1}.

Another possibility is to allow the solution
to have jumps at the points $t_k$. This leads
to \emph{Impulsive Differential Equations with
Piecewise Constant Argument of Generalized Type}
(IDEPCAG), which can be written as
\begin{alignat}{2}
x'(t)
&=f(t,x(t),x(\gamma(t))),
\qquad &&t\neq t_k,\nonumber\\
\Delta x\big|_{t=t_k}
&=Q_k(x(t_k^-)),
\qquad &&t=t_k,\nonumber
\end{alignat}
where
\[
\Delta x|_{t=t_k}=x(t_k)-x(t_k^-),
\qquad
x(t_k^-)=\lim_{\substack{t\to t_k\\t<t_k}}x(t).
\]
Thus, the solution follows a differential equation
between consecutive impulse times, while its
value changes according to the prescribed jump
condition at each $t_k$. Some qualitative
properties of these systems, including stability,
have been studied in \cite{AKH3}.

An interesting area in which these ideas can
be applied is the mathematical study of neural
networks. Their development has led to many
different models, depending on the type of
network and the behavior one wishes to describe.
Among the best-known contributions are the
recurrent neural networks introduced by Hopfield
\cite{HOPFIELD} and the cellular neural networks
(CNNs) developed by Chua and Yang
\cite{CHUAYANG1988}. These models have been
studied both from a theoretical point of view
and in connection with applications such as
signal processing, image processing, and
pattern recognition.

Neural network models can differ considerably
from one another. Some describe the evolution
of the cells in continuous time, while others
use discrete time. The behavior of the cells
may also be modeled through deterministic or
probabilistic equations. Each approach has
its own advantages, and the choice of a model
usually depends on the problem being considered.

In the continuous and deterministic case,
a commonly used model is the nonlinear system
\begin{equation}
\frac{dx_i(t)}{dt}
=-a_i(t)x_i(t)
+\sum_{j=1}^{m}b_{ij}f_j(x_j(t))
+d_i(t),
\qquad i=1,\ldots,m,
\label{1}
\end{equation}
where $m$ is the number of cells and $x_i(t)$
represents the state, or membrane potential,
of the $i$th cell at time $t$. The coefficient
$a_i(t)$ describes the rate at which this
state decays in the absence of other effects,
and $b_{ij}$ measures the connection from
the $j$th cell to the $i$th cell. The function
$f_j$ describes the output of the $j$th cell,
while $d_i(t)$ represents an external input
acting on the $i$th cell.

This system is related to the electrochemical
properties of neural networks and can be
obtained by considering circuits with resistive
elements. Many aspects of its dynamics have
been studied, including stability, periodic
and almost periodic solutions, oscillations,
chaos, and bifurcations; see
\cite{HUANG2,HUANG1}.

Motivated by these developments, several generalized versions of~\eqref{1} have been proposed. In some experimental settings, the trajectories of the state variables exhibit abrupt changes, suggesting that impulsive differential equations (IDEs) provide a more appropriate mathematical framework. These considerations have led to the study of neural network models combining impulsive dynamics with differential equations with piecewise constant arguments. Such IDE--DEPCA models have been considered in the mathematical literature; see, for example, \cite{AKH3,AKH8,BEREK1}.

\subsection{Cellular neural networks with piecewise constant argument}
Cellular neural networks of the form \eqref{1} have been studied
extensively in both the DEPCA and IDEPCA settings. In addition
to the works cited in \cite{Pinto_ejqtde}, we discuss the following
contributions.

Chiu \cite{Chiu_2022} studied the following autonomous impulsive
neural network with a generalized piecewise constant argument:
\begin{alignat}{2}
 x_{i}^{\prime}(t)  & =-a_{i}x_{i}(t)+\sum_{j=1}^{m}b_{ij}f_{j}(x_{j}(t))+\sum_{j=1}^{m}c_{ij}g_{j}(x_{j}(\gamma(t)))+d_{i},\qquad && t\neq t_{k},\label{kuo2023}\\
\Delta x_{i}|_{t=t_{k}} & =J_{i,k}(x_{i}(t_{k}^{-})),\qquad && t=t_{k},\nonumber
\end{alignat}
where $\gamma(t)=\zeta_k$ for $t\in I_k=[t_k,t_{k+1}),$ with $t_k\leq \zeta_k<t_{k+1}.$
Existence, uniqueness, and global exponential stability of the equilibrium state were established for this system.\\

Motivated by \cite{AKH8}, Pinto et al.
\cite{Pinto_ejqtde} studied the existence,
uniqueness, and stability of periodic solutions
for the impulsive Hopfield-type neural network
with a piecewise constant argument
\begin{alignat}{2}
y_i'(t)
&=-a_i(t)y_i(t)
+\sum_{j=1}^{m}b_{ij}(t)f_j(y_j(t))
+\sum_{j=1}^{m}c_{ij}(t)g_j(y_j(\gamma(t)))
+d_i(t),\qquad && t\neq t_k, \nonumber\\
\Delta y_i\big|_{t=t_k}
&=-q_{i,k}y_i(t_k^-)
+I_{i,k}(y_i(t_k^-))
+e_{i,k},\qquad && t=t_{k},\nonumber
\end{alignat}
where $i=1,\ldots,m$. In this model, the piecewise constant argument is 
$\gamma(t)=t_k$ for $t\in I_k=[t_k,t_{k+1}),$ with $k\in\mathbb{N}.$
Under periodicity and Lipschitz assumptions, the authors used
period averages, a Green function, and the Banach fixed-point
theorem to establish the existence and uniqueness of a periodic
solution and to derive sufficient conditions for its exponential
attractivity. Their setting corresponds to the left-endpoint
argument on each interval.\\

More recently, Oztepe and Karako\c{c} \cite{OztepeKarakoç2026}
studied the stability of a Hopfield-type neural network with conformable
derivatives and multiple piecewise constant arguments of the form
$[t-j]$. More precisely, they considered the system
\begin{equation}
T_n^\alpha(x_i)(t)
=
-a_i x_i(t)
+\sum_{j=1}^{m} b_{ij}
f_j\bigl(x_j([t-j])\bigr)
+\theta_i,
\qquad i=1,\ldots,m,
\end{equation}
where $0<\alpha\leq1$. After translating the unique equilibrium
$\bar{x}$ to the origin by means of
\[
y_i(t)=x_i(t)-\bar{x}_i,
\]
the system takes the form
\begin{equation}
T_n^\alpha(y_i)(t)
=
-a_i y_i(t)
+\sum_{j=1}^{m}b_{ij}
g_j\bigl(y_j([t-j])\bigr).
\end{equation}
On each interval $I_n=[n,n+1),$ with $n\in\mathbb{N},$ the piecewise constant arguments satisfy
$[t-j]=n-j$, which allows the solution to be written as
\begin{equation}
y_i(t)
=
e^{-\frac{a_i}{\alpha}(t-n)^\alpha}y_i(n)
+
\frac{1-e^{-\frac{a_i}{\alpha}(t-n)^\alpha}}{a_i}
\sum_{j=1}^{m}b_{ij}g_j\bigl(y_j(n-j)\bigr).
\end{equation}
Consequently, evaluating at $t=n+1$ yields the associated difference
system
\begin{equation}
y_i(n+1)
=
e^{-a_i/\alpha}y_i(n)
+
\frac{1-e^{-a_i/\alpha}}{a_i}
\sum_{j=1}^{m}b_{ij}g_j\bigl(y_j(n-j)\bigr).
\end{equation}
Using this discrete representation together with Lyapunov-type
arguments and a discrete Halanay inequality, sufficient conditions for
the asymptotic and exponential stability of the equilibrium are
obtained. This discrete reduction describes the dynamics at the switching instants. Together with the continuous evolution between consecutive switching times, it provides a description of the full hybrid dynamics.\\

 Finally, Chiu \cite{CHIU2026MMA} considered the recurrent neural
network with generalized piecewise constant delay
\begin{equation}
x_j'(t)
=
-a_j(t)x_j(t)
+\sum_{k=1}^{m}
b_{jk}(t)f_k(x_k(t))
+\sum_{k=1}^{m}c_{jk}(t)g_k(x_k(\beta(t)))
+I_j(t),
\qquad j=1,\ldots,m,
\end{equation}
where $\beta(t)=t_k$, for $t\in I_k=[t_k,t_{k+1}).$ Under suitable assumptions on Lipschitz continuity, solvability, and periodicity, the
author established the existence of periodic solutions by means of
Mawhin's coincidence degree theory. Global exponential stability and
attractivity of the periodic solution were subsequently obtained through
Lyapunov-type arguments and differential inequalities adapted to the
generalized piecewise constant delay. The model considered in \cite{CHIU2026MMA} is non-impulsive. Although it incorporates a generalized piecewise constant delay, it does not include the jump dynamics at the impulse times of the Hopfield-type system studied here.

\subsection{Aim of the paper}

The main objective of this paper is to establish
sufficient conditions for the existence of a unique
$\omega$-periodic solution and its global exponential
attractivity in impulsive and nonimpulsive Hopfield-type neural
networks with generalized piecewise constant
arguments. A particular emphasis is placed on
obtaining explicit conditions that are relatively
straightforward to verify in concrete applications,
despite the presence of both impulsive effects
and delayed or advanced arguments.

Our analysis builds on the framework developed by
Pinto et al. \cite{Pinto_ejqtde}, which is recovered
when $\gamma(t)=t_k$. We extend this framework to
$\gamma(t)=\zeta_k$, with
$t_k\leq\zeta_k\leq t_{k+1}$, allowing both
delayed and advanced arguments. We also correct
some estimates in the original analysis and adapt
the arguments to this more general setting.

More precisely, we consider the following IDEPCAG (DEPCAG)
Hopfield-type neural network:

\begin{subequations}\label{eq:hopf_impulsivo}
\begin{alignat}{2}
x_i'(t)
&=-a_i(t)x_i(t)
+\sum_{j=1}^{m}b_{ij}(t)f_j(x_j(t))
+\sum_{j=1}^{m}c_{ij}(t)g_j(x_j(\gamma(t)))
+d_i(t),\qquad && t\neq t_k,
\label{2}\\
\Delta x_i\big|_{t=t_k}
&=-q_{i,k}x_i(t_k^-)
+I_{i,k}(x_i(t_k^-))
+e_{i,k},\qquad && t=t_k
\label{eq:2_imp_cond}
\end{alignat}
for $i=1,2,\ldots,m$, where $m$ denotes the number of neurons in the
network.

\begin{remark}\hfill
\begin{itemize}
\item If $\zeta_k=t_{k+1}$, the deviating value
$x(\gamma(t))$ is understood as
$x(t_{k+1}^-)$ for $t\in I_k$.
\item We take $t_0=0$ as the first mesh point and
prescribe $x(0)=x_0$. The jump conditions are
imposed only for $k\in\mathbb N$. This choice
entails no loss of generality, since any initial
mesh point can be shifted to the origin without
changing the $(\omega,p)$-periodicity.
\end{itemize}
\end{remark}

We assume that, for some constant $\bar{\theta}>0$,
\begin{equation}
0<t_{k+1}-t_k\leq\bar{\theta},
\qquad k\in\mathbb{N}_0,
\qquad
\lim_{k\to\infty}t_k=+\infty.
\label{eq:2_seq_tk}
\end{equation}

At each impulse time $t_k$, $k\in\mathbb N$,
the jump of $x_i$ is defined by
\begin{equation}
\Delta x_i\big|_{t=t_k}
=x_i(t_k)-x_i(t_k^-),
\qquad
x_i(t_k^-)
=\lim_{\substack{t\to t_k\\ t<t_k}}x_i(t).
\label{eq:jump}
\end{equation}
\end{subequations}

Following \cite{Pinto_ejqtde}, the functions and parameters appearing in
\eqref{2}--\eqref{eq:jump} have the following interpretations:
\begin{enumerate}
    \item[(i)] The function $x_i(t)$ represents the state of the $i$th unit
    at time $t$; in particular, $x_i$ typically denotes the potential of
    the $i$th neuron of the network.

    \item[(ii)] The function $a_i(t)>0$ and the parameters
    $0\leq q_{i,k}<1$ represent the resetting rates of the potential of the
    $i$th unit.
    \item[(iii)] The functions $f_j(x_j(t))$ and
    $g_j(x_j(\gamma(t)))$ represent the activation responses associated
    with the incoming potential from unit $j$ to unit $i$.

    \item[(iv)] The functions $b_{ij}(t)$ and $c_{ij}(t)$ represent the
    connection weights from unit $j$ to unit $i$.

    \item[(v)] The function $d_i(t)$ represents the external input applied
    to the $i$th unit.

    \item[(vi)] The function $I_{i,k}(x_i(t_k^-))$ represents the impulsive
    activation acting on the $i$th unit at the impulse time $t_k$, where
    \[
    I_{i,k}(x_i(t_k^-))
    =
    \lim_{\substack{t\to t_k\\ t<t_k}}
    I_{i,k}(x_i(t)).
    \]

    \item[(vii)] The parameter $e_{i,k}$ represents the external input
    applied to the $i$th unit at the impulse time $t_k$.
\end{enumerate}

Here, $\N$ and $\R_0^+=[0,\infty)$ denote the sets of natural and
nonnegative real numbers, respectively. Note that
\eqref{eq:hopf_impulsivo} can be regarded as a perturbation of the
nonhomogeneous linear impulsive system
\begin{alignat}{2}
x_i'(t)
&=-a_i(t)x_i(t)+d_i(t),\qquad && t\neq t_k,
\label{imp_no_perturbado}\\
\Delta x_i\big|_{t=t_k}
&=-q_{i,k}x_i(t_k^-)+e_{i,k},\qquad && t=t_k.\nonumber
\end{alignat}

Let $X\subset\mathbb{R}_0^+$ and $Y\subset\mathbb{R}$. We call $\mathcal{PC}(X, Y)$ the set of all functions $x: X\to Y$ that are continuous for $t\neq t_k$ and right-continuous, with possible jump discontinuities at the prescribed impulse times $t=t_k$. 

\begin{definition}[DEPCAG solution]\hfill\\
Let $x=(x_1,\ldots,x_m):\mathbb R_0^+\to\mathbb R^m$. The continuous function $x$ is called a DEPCAG solution of the differential part of \eqref{eq:hopf_impulsivo}~if, for each $i=1,\ldots,m$, the following conditions hold:
\begin{itemize}
\item[(i)] Each component $x_i(t)$ is continuous on each $I_k=[t_k,t_{k+1})$, 
where $x_i'(t)$ exists at each $t\in \mathbb{R}_0^+$ with the possible exception at the times $t_{k}$, where one-sided derivatives exist.
\item[(ii)] The ordinary differential equation \eqref{2}
 holds on every interval $I_{k}$, where $\gamma(t)=\zeta_k$.
\end{itemize}
\end{definition}

\begin{definition}[IDEPCAG solution]\hfill\\
Let $x=(x_1,\ldots,x_m):\mathbb R_0^+\to\mathbb R^m$. The piecewise continuous function $x$ is a solution of \eqref{eq:hopf_impulsivo}~if, for each $i=1,\ldots,m$, the following conditions hold:
\begin{itemize}
\item[(i)] Each component $x_i(t)$ is continuous on $I_k=[t_k,t_{k+1})$ with possible first kind discontinuities at $t_k, \,k\in \mathbb{N}$, 
where $x_i'(t)$ exists at each $t\in \mathbb{R}_0^+$ with the possible exception at the times $t_{k}$, where one-sided derivatives exist.
\item[(ii)] The ordinary differential equation \eqref{2}
 holds on every interval $I_{k}$, where $\gamma(t)=\zeta_k$.
\item[(iii)]
For $t=t_{k}$, the impulsive condition 
$$\Delta x_i\big|_{t=t_k}=-q_{i,k}x_i(t_k^-)+I_{i,k}(x_i(t_k^-))
+e_{i,k}$$
holds. I.e., $x_i(t_{k})=(1-q_{i,k})x_i(t_{k}^{-})+I_{i,k}(x_i(t_k^-))
+e_{i,k}$, where $x_i(t_{k}^{-})$ denotes the left-hand limit
of the function $x_i$ at $t_k$, with $k\in\mathbb{N}$.
\end{itemize}
\end{definition}

\subsection{General assumptions}
We collect below the assumptions that will be used throughout the paper
for system \eqref{eq:hopf_impulsivo}.

\begin{enumerate}

\item[{\rm(H1)}]
The functions $a_{i},b_{ij},c_{ij},d_{i}$ take real values and are locally integrable and \(\omega\)-periodic, with $\omega>0$ and $a_i(t)>0,$ for all $t\in [0,\infty)$ and for all $i\in\{1,\ldots,m\}$. We also assume that there exists
$p\in\N$ such that the sequences
$\{q_{i,k}\}_{k\in\N}$,
$\{e_{i,k}\}_{k\in\N}$, and $\{I_{i,k}\}_{k\in\N}$ satisfy
\begin{align*}
&
\Omega:=\{k:0<t_k\leq\omega\},
\qquad \#\Omega=p,
\\
&
q_{i,k+p}=q_{i,k}, \text{ with } 0\leq q_{i,k}<1,
\\
&
e_{i,k+p}=e_{i,k},
\qquad I_{i,k+p}=I_{i,k},
\quad \forall k\in\N,\,\,\forall i\in\{1,\ldots,m\}.
\end{align*}
The sequences $\{t_k\}_{k\in\mathbb{N}_0}$ and $\{\zeta_k\}_{k\in\mathbb{N}_0}$ satisfy the
$(\omega,p)$-property
\[
t_{k+p}=t_k+\omega,
\qquad
\zeta_{k+p}=\zeta_k+\omega,
\qquad k\in\mathbb N_0.
\]

\item[{\rm(H2)}]
For each $i=1,\ldots,m$, the function $a_i(t)$ and the sequence
$\{q_{i,k}\}_{k\in\N}$ satisfy
\begin{align*}
\displaystyle
\Phi_i(\omega,0)
=
\prod_{k=1}^{p}\Big(1-q_{i,k}\Big)
\exp\left({-\int_0^{\omega}a_{i}(u)du}\right)
\neq1.
\end{align*}
It is clear that {\rm(H1)} implies {\rm(H2)}. This corresponds to the noncritical case for periodic systems.

\item[{\rm(H3)}]
For every $j=1,\ldots,m$, the functions $f_j$ and $g_j$ satisfy the
Lipschitz estimates
\begin{align*}
|f_{j}(u)-f_{j}(v)|
&\leq L_{j}|u-v|,
\\
|g_{j}(u)-g_{j}(v)|
&\leq\bar{L}_{j}|u-v|,
\end{align*}
for all $u,v\in\R$, where $L_j,\bar L_j>0$.

\item[{\rm(H4)}]
For each $i=1,\ldots,m$ and $k\in\N$, there exists $l_{i,k}\geq 0$ such
that
\begin{align*}
|I_{i,k}(u)-I_{i,k}(v)|
\leq
l_{i,k}|u-v|,
\qquad
\forall u,v\in\R.
\end{align*}
Consistently with the periodicity condition
$I_{i,k+p}=I_{i,k}$ in {\rm (H1)}, the Lipschitz constants are chosen
$p$-periodic, that is,
\[
l_{i,k+p}=l_{i,k},
\qquad k\in\N.
\]

\item[{\rm(H5)}]
The functions $f_j$, $g_j$, and $I_{i,k}$ satisfy
\[
f_j(0)=g_j(0)=I_{i,k}(0)=0,
\]
and conditions {\rm (H3)} and {\rm (H4)} hold whenever
$|u|,|v|\leq R$, for some $R>0$.

\item[{\rm(H6)}]
There exists $\sigma>0$ such that
\begin{align*}
\int_s^t a_i(u)\,du
+\sum_{s<t_k\leq t}\ln(1+q_{i,k})
\geq \sigma(t-s),
\qquad
0\leq s\leq t,\quad
i=1,\ldots,m.
\end{align*}

\end{enumerate}

Some of the results below will require three additional assumptions.

\begin{enumerate}

\item[{\rm(H7)}]
Set
\begin{equation*}
\rho
=
\sup_{n\in\N_0}
\int_{t_{n}}^{t_{n+1}}
\Big(\widetilde{b}(s)+\widetilde{c}(s)\Big)ds.
\end{equation*}
We assume that
\[
\rho<1,
\]
where
\begin{equation}
\widetilde{b}(s)
=
\sum_{i=1}^{m}\sum_{j=1}^{m}|b_{ij}(s)|L_{j},
\qquad
\widetilde{c}(s)
=
\sum_{i=1}^{m}\sum_{j=1}^{m}|c_{ij}(s)|\overline{L_{j}}.
\label{eq:notationhatbc}
\end{equation}
Here, $L_j$ and $\overline{L_j}$ are the constants introduced in
{\rm (H3)}.

\item[{\rm(H8)}]
The following inequality is assumed:
\begin{equation}
\kappa=\mathcal{H}
\left(
\omega\mathcal{M}(\widetilde{b}+\widetilde{c})
+
p\mathcal{M}(\widetilde{l})
\right)
<1.
\label{h8}
\end{equation}
Here, $\mathcal H$ is the uniform norm of the Green functions
$\mathcal H_i$ given in \eqref{h}, as specified below. The corresponding mean values of $\widetilde{b},\widetilde{c}$ and $\widetilde{l}$ are 
\[
\mathcal{M}(\widetilde{b})
=
\dfrac{1}{\omega}\int_0^{\omega}\widetilde{b}(u)du,
\qquad
\mathcal{M}(\widetilde{c})
=
\dfrac{1}{\omega}\int_0^{\omega}\widetilde{c}(u)du,
\qquad 
\mathcal{M}(\widetilde{l})
=
\dfrac{1}{p}\sum_{k=1}^{p}\widetilde{l}_{k},
\]
where
\begin{equation}
\widetilde{l}_{k}
=
\sum_{i=1}^{m}l_{i,k}.
\label{eq:notationhatl}
\end{equation}

\item[{\rm(H9)}]
The constant $\sigma>0$ from {\rm (H6)} is further assumed to satisfy 
\begin{equation*}
\sigma>\mathcal{M}
\Big(
\widetilde{b}
+
\widetilde{c}\left(\tfrac{e^{\bar{\theta}\sigma}}{1-\vartheta}\right)
\Big)
+
\frac{p}{\omega}\mathcal{M}
\Big(
\ln(1+\widetilde{l})
\Big)
.
\end{equation*}
The constant $\vartheta$ is defined by
\begin{align*}
\vartheta
&=
\sup_{k\in\mathbb N_0}
\int_{t_k}^{\zeta_k}
\left(
\widetilde b(s)
+e^{\sigma\bar\theta}\widetilde c(s)
\right)\,ds
<1.
\end{align*}
\end{enumerate}

\begin{remark}\hfill\\
Under {\rm (H1)}, condition {\rm (H8)}
implies {\rm (H7)}, since
\[
\rho
\leq\omega\mathcal M(\widetilde b+\widetilde c)
\leq\kappa<1.
\]
\end{remark}

\begin{remark}
\label{rem:hyp}\hfill\\
A few comments on these assumptions will be useful.
\begin{itemize}
\item[(a)]
Assumption {\rm (H1)} implies
\[
\gamma(t+\omega)=\gamma(t)+\omega,
\qquad t\geq0.
\]
Indeed, if $t\in I_k$ for some $k\in\mathbb N_0$, then
$t+\omega\in I_{k+p}$ by {\rm (H1)}. Hence,
\begin{equation}
\gamma(t+\omega)
=\zeta_{k+p}
=\zeta_k+\omega
=\gamma(t)+\omega,
\qquad t\geq0.
\label{eq:gamma-periodic}
\end{equation}
\item[(c)]
The role of assumption {\rm (H6)} can be seen directly from the
evolution operator of the homogeneous linear impulsive system,
\[
\Phi_i(t,s)
=
\prod_{s<t_k\leq t}(1-q_{i,k})
\exp\left(-\int_s^t a_i(u)\,du\right).
\]
Since $0\leq q_{i,k}<1$,
\[
1-q_{i,k}
\leq
\frac{1}{1+q_{i,k}},
\]
and hence
\begin{align*}
|\Phi_i(t,s)|
&\leq
\prod_{s<t_k\leq t}\frac{1}{1+q_{i,k}}
\exp\left(-\int_s^t a_i(u)\,du\right)
\\
&=
\exp\left(
-\left(
\int_s^t a_i(u)\,du
+
\sum_{s<t_k\leq t}\ln(1+q_{i,k})
\right)
\right).
\end{align*}
Therefore, {\rm (H6)} yields
\[
|\Phi_i(t,s)|
\leq e^{-\sigma(t-s)}.
\]
Thus, the term $\ln(1+q_{i,k})$ should be understood as a convenient
estimate for the impulsive contribution, rather than as the exact
logarithmic contribution of the multiplier $1-q_{i,k}$.

This point is also consistent with the classical stability theory for
linear impulsive systems with constant commuting coefficients; see
Samoilenko and Perestyuk~\cite{SAMOILENKO}. In that setting, the
continuous and impulsive effects are balanced through the asymptotic
impulse frequency
\[
\nu=\lim_{w\to\infty}\frac{i(t,t+w)}{w},
\]
where $i(s,t)$ denotes the number of impulsive moments in $(s,t]$.
Under {\rm (H1)}, there are $p$ impulses in each period of length
$\omega$, so that
$\nu=\frac{p}{\omega}.$
Accordingly, in the scalar constant-coefficient case, {\rm (H6)}
reduces to
\[
a(t-s)+i(s,t)\ln(1+q)
\geq \sigma(t-s),
\]
and its long-time average gives
\[
a+\frac{p}{\omega}\ln(1+q)
\geq \sigma>0.
\]

For periodic coefficients, the same idea is expressed through the
mean continuous dissipation
\[
\mathcal M(a_i)
=
\frac{1}{\omega}\int_0^\omega a_i(u)\,du,
\]
together with the mean impulse frequency $p/\omega$. In particular,
\[
\frac{1}{\omega}
\left(
\int_0^\omega a_i(u)\,du
+
\sum_{k=1}^{p}\ln(1+q_{i,k})
\right)
\]
measures the total continuous and impulsive dissipative contribution
over one period.
\item[(d)]
The nonimpulsive case (DEPCAG) is also included in our
framework. Indeed, if
$q_{i,k}=0$, $I_{i,k}\equiv 0$, and $e_{i,k}=0$
for all $i$ and $k$, the jump conditions vanish,
and the system reduces to a nonimpulsive
Hopfield-type neural network with a generalized
piecewise constant argument.

In this case, the Green function reduces to
that of the corresponding nonimpulsive linear
system, and the sufficient conditions for
existence, uniqueness, and global exponential
attractivity remain applicable whenever the
corresponding assumptions are satisfied.
\end{itemize}
\end{remark}

\section{Existence and uniqueness of solutions for
\texorpdfstring{(\ref{eq:hopf_impulsivo})}{the Hopfield system}}
\label{sec:exist}
Throughout this section, the coefficients
$a_i$, $b_{ij}$, $c_{ij}$, and $d_i$
are assumed to be locally integrable.
\subsection{A Gronwall-type inequality with piecewise constant arguments}

We begin with an auxiliary Gronwall--Bellman inequality adapted to
generalized piecewise constant arguments. This estimate will be used
repeatedly in the analysis below; see \cite{Pinto_ejqtde,CTP2019}.

\begin{lemma}{\cite{Pinto_ejqtde,CTP2019}}
\label{lem:gronw_lem1}\hfill\\ 
Let $I$ be an interval and let
$u,\eta_{1},\eta_{2}:I\subset\R\to\R_{0}^{+}$.
Assume that $u$ is piecewise continuous, that $\eta_{1}$ and $\eta_{2}$ are
locally integrable, and that
$\eta:\{t_k\}\to\R^{+}_{0}$.
Let $\gamma(t)$ be a generalized piecewise constant argument, that is,
\[
\gamma(t)=\zeta_k,
\qquad
t\in I_k=[t_k,t_{k+1}),
\]
where
\[
t_k\leq \zeta_k\leq t_{k+1},
\qquad k\in\N_0.
\]
Assume that
\[
\nu
=
\sup_{k\in\mathbb N_0}
\int_{t_k}^{\zeta_k}
\left(\eta_1(s)+\eta_2(s)\right)\,ds
<1.
\]
Suppose, moreover, that for
$t_0\leq\tau\leq t$ in $I$, the function $u$ satisfies
\[
u(t)
\leq
u(\tau)
+
\int_{\tau}^{t}
\left[
\eta_1(s)u(s)+\eta_2(s)u(\gamma(s))
\right]\,ds
+
\sum_{\tau<t_k\leq t}
\eta(t_k)u(t_k^-).
\]
Then, for every $t\geq t_0$,
\begin{align*}
u(t)
&\leq
\left(
\prod_{ t_0 < t_{k}\leq t}
\left(1+\eta(t_{k})\right)
\right)
\exp\left(
\int_{ t_0}^{t}
\left(
\eta_{1}(s)+\frac{\eta_{2}(s)}{1-\nu}
\right)ds
\right)
u( t_0),
\\
u(\zeta_{k})
&\leq
(1-\nu)^{-1}u(t_{k}).
\end{align*}
When $\zeta_k=t_{k+1}$, the values
$u(\gamma(t))$, for $t\in I_k$, and
$u(\zeta_k)$ in the conclusion are
understood as $u(t_{k+1}^-)$.
\end{lemma}

\subsection{Local solvability on an interval between consecutive impulses}

Fix $r\in\mathbb N_0$ and consider
$I_r=[t_r,t_{r+1})$. Starting from the state $x(t_r)=x_0$,
we keep the deviating point $\zeta_r\in[t_r,t_{r+1}]$ unchanged.
The local initial-value problem is
\begin{align}
\begin{array}{l}
\displaystyle
x_i'(t)
=-a_i(t)x_i(t)
+\sum_{j=1}^{m}b_{ij}(t)f_j(x_j(t))
+\sum_{j=1}^{m}c_{ij}(t)g_j(x_j(\zeta_r))
+d_i(t),
\\[2mm]
\displaystyle
x_i(t_r)=x_i^0.
\end{array}
\label{3}
\end{align}

 Problem~\eqref{3} admits the equivalent integral formulation
\begin{align}
w_i(t)
&=\mathcal T_i(w,t_r,x_0)(t),
\qquad
w(t)=(w_1(t),\ldots,w_m(t)),
\qquad
t\in [t_r,t_{r+1}],
\label{eq:integ_equation_3}
\end{align}
where
\begin{equation}
\label{eq:integ_equation_3_H}
\begin{split}
\MoveEqLeft[3]\mathcal T_i(w,t_r,x_0)(t)\\
={}&
\exp\Big({-\int_{t_r}^{t}a_{i}(u)du}\Big)x_{i}^{0}
\\
&+
\int_{t_r}^{t}
\exp\Big({-\int_{s}^{t}a_{i}(u)du}\Big)
\left(
\sum_{j=1}^{m}b_{ij}(s)f_{j}(w_{j}(s))
+\sum_{j=1}^{m}c_{ij}(s)g_{j}(w_{j}(\zeta_{r}))
+d_{i}(s)
\right)ds.
\end{split}
\end{equation}

The next two lemmas provide the local uniqueness and existence results
needed for the construction of solutions. They follow the arguments in
\cite{Pinto_ejqtde}, adapted here to the notation and generalized
piecewise constant argument of the present system.

\begin{lemma}[Uniqueness {\cite{Pinto_ejqtde}}]
\label{lem:localuniqueness}\hfill\\
Fix $r\in\mathbb N_0$ and
$x_0=(x_1^0,\ldots,x_m^0)^T\in\R^m$.
Assume that a solution of \eqref{3} with $x(t_r)=x_0$ exists.
If {\rm (H3)}, {\rm (H6)}, and {\rm (H7)} hold, then the solution
$x(t)=x(t,t_r,x_0)=(x_1(t),\ldots,x_m(t))^T$
of \eqref{3} is unique.
\end{lemma}
\begin{proof}
    Let \(x\) and \(y\) be two solutions of \eqref{3} with the same initial value at \(t_r\). Under the stated assumptions, applying Lemma \ref{lem:gronw_lem1} to the norm of their difference implies that the two solutions coincide.
\end{proof}
\begin{lemma}[Local existence {\cite{Pinto_ejqtde}}]
\label{lem:localexistence}\hfill\\
Assume that {\rm (H3)}, {\rm (H6)}, and {\rm (H7)} hold. Then, for every
$r\in\mathbb N_0$ and every
$x_0=(x_1^0,\ldots,x_m^0)^T\in\R^m$,
there exists a solution
$x(t)=x(t,t_r,x_0)=(x_1(t),\ldots,x_m(t))^T$
of \eqref{3} on $[t_r,t_{r+1}]$ satisfying
$x(t_r)=x_0.$
\end{lemma}

\begin{proof}
It is enough to prove that the integral equation
\eqref{eq:integ_equation_3} has a solution on $[t_r,t_{r+1}]$.
We use the norm
\[
\|x\|_0
=
\max_{t\in[t_r,t_{r+1}]}\|x(t)\|_1,
\qquad
\|x(t)\|_1
=
\sum_{i=1}^{m}|x_i(t)|.
\]

We construct a sequence
\[
\{x^n\}_{n\geq0},
\qquad
x^n=(x_1^n,\ldots,x_m^n)^T,
\]
by setting
\[
x_i^0(t)
=
\mathcal T_i(0,t_r,x_0)(t),
\]
and, recursively,
\[
x_i^{n+1}(t)
=
\mathcal T_i(x^n,t_r,x_0)(t),
\qquad
n\geq0,
\]
where $\mathcal T_i$ is defined in
\eqref{eq:integ_equation_3_H}.

By {\rm (H3)}, for $n\geq1$ and
$t\in[t_r,t_{r+1})$, we obtain
\begin{align*}
\|x^{n+1}(t)-x^n(t)\|_1
\leq{}&
\sum_{i=1}^{m}
\int_{t_r}^{t}
\exp\left(-\int_s^t a_i(u)\,du\right)
\\
&\times
\Bigg(
\sum_{j=1}^{m}
|b_{ij}(s)|L_j
|x_j^n(s)-x_j^{n-1}(s)|
\\
&\qquad+
\sum_{j=1}^{m}
|c_{ij}(s)|\overline L_j
|x_j^n(\zeta_r)-x_j^{n-1}(\zeta_r)|
\Bigg)\,ds,
\end{align*}
with $\zeta_r\in[t_r,t_{r+1}]$.\\

Using the notation introduced in
\eqref{eq:notationhatbc}, we then have
\begin{align*}
\|x^{n+1}(t)-x^n(t)\|_1
&\leq
\|x^n-x^{n-1}\|_0
\int_{t_r}^{t}
e^{-\sigma(t-s)}
\bigl(\widetilde b(s)+\widetilde c(s)\bigr)\,ds
\\
&\leq
\|x^n-x^{n-1}\|_0
\int_{t_r}^{t_{r+1}}
\bigl(\widetilde b(s)+\widetilde c(s)\bigr)\,ds
\\
&\leq
\rho\,
\|x^n-x^{n-1}\|_0.
\end{align*}
Passing to the limit as $t\to t_{r+1}^{-}$ gives the same estimate at
$t=t_{r+1}$. Hence,
\[
\|x^{n+1}-x^n\|_0
\leq
\rho\,
\|x^n-x^{n-1}\|_0.
\]
By induction,
\[
\|x^{n+1}-x^n\|_0
\leq
\rho^n
\|x^1-x^0\|_0.
\]

Since $\rho<1$ by {\rm (H7)}, the series
\[
\sum_{n=0}^{\infty}
\|x^{n+1}-x^n\|_0
\]
converges. Thus, $\{x^n\}$ is a Cauchy sequence and converges uniformly
on $[t_r,t_{r+1}]$ to a function $x$.

Passing to the limit in the recursive relation gives
\[
x_i(t)
=
\mathcal T_i(x,t_r,x_0)(t),
\qquad
i=1,\ldots,m.
\]
Therefore, $x$ satisfies \eqref{eq:integ_equation_3}, and hence it is a
solution of \eqref{3} on $[t_r,t_{r+1}]$ satisfying
$x(t_r)=x_0$.\\
At the endpoint $t_{r+1}$, the value obtained from the local construction is understood as the left-hand limit
$x(t_{r+1}^-)$.
\end{proof}

Combining the previous two lemmas, the local solution can be continued
through the successive impulse times. Starting from an initial state 
$x(0)=x_0$, the solution is first constructed up to the next impulse. The
jump condition then determines the state at the impulse time uniquely, and this
value becomes the initial condition for the differential problem on the
following interval. Repeating this procedure extends the solution to
all subsequent times.

\begin{theorem}[{\cite{Pinto_ejqtde}}]\hfill\\
Assume that {\rm (H3)--(H4)}, {\rm (H6)}, and {\rm (H7)} hold. Then, for
every initial state
$
x_0\in\R^m,
$
there exists a unique solution
\[
x(t)
=
x(t, 0,x_0)
=
(x_1(t),\ldots,x_m(t))^T,
\qquad
t\geq 0,
\]
of \eqref{eq:hopf_impulsivo} satisfying
\[
x(0)=x_0.
\]
\end{theorem}

\subsection{Integral representations associated with
\texorpdfstring{(\ref{eq:hopf_impulsivo})}{the Hopfield system}}

We next describe an integral representation of the solutions of
\eqref{eq:hopf_impulsivo}. It is obtained by combining the continuous
evolution between consecutive impulse times with the contributions
generated by the jumps.

\begin{lemma}
\label{lem:integ_equation_1}\hfill\\
A function
\[
x(t)
=
x(t,0,x_{0})
=
(x_{1}(t),\ldots,x_{m}(t))^{T}
\]
solves \eqref{eq:hopf_impulsivo} on $[0,\infty)$ precisely when, for
each $i=1,\ldots,m$ and every $t\geq 0$, its components satisfy
\begin{align*}
x_{i}(t)
={}&
x_{i}^{0}
+
\int_{0}^{t}
\left(
-a_{i}(s)x_i(s)
+\sum_{j=1}^{m}b_{ij}(s)f_{j}(x_{j}(s))
+\sum_{j=1}^{m}c_{ij}(s)g_j(x_j(\gamma(s)))
+d_{i}(s)
\right)ds
\\
&+
\sum_{0<t_k\leq t}
\left(
-q_{i,k}x_{i}(t_{k}^{-})
+
\widetilde{h}_{i,k}\left(x_{i}(t_{k}^{-})\right)
\right),
\end{align*}
where
\[
\widetilde{h}_{i,k}\left(x_{i}(t_{k}^{-})\right)
=
I_{i,k}(x_{i}(t_{k}^{-}))
+
e_{i,k}.
\]
\end{lemma}

\begin{lemma}
\label{lem:integ_equation_2}\hfill\\
Let
\[
x(t)
=
x(t,0,x_{0})
=
(x_{1}(t),\ldots,x_{m}(t))^{T}.
\]
The function $x$ solves \eqref{eq:hopf_impulsivo} on $[0,\infty)$ if
and only if, for $i=1,\ldots,m$ and $t\geq 0$,
\begin{align*}
x_{i}(t)
={}&
\Phi_i(t,0)x_{i}^{0}
+
\int_{0}^{t}
\Phi_i(t,s)
\left(
\sum_{j=1}^{m}b_{ij}(s)f_{j}(x_{j}(s))
+
\sum_{j=1}^{m}c_{ij}(s)g_j(x_j(\gamma(s)))
+
d_{i}(s)
\right)ds
\\
&+
\sum_{0<t_k\leq t}
\Phi_i(t,t_k)
\widetilde{h}_{i,k}(x_i(t_{k}^-)).
\end{align*}
Here, $k=k(t)$ denotes the unique $k\in\N_0$ for which
$t\in[t_k,t_{k+1})$, and the evolution operator is given by
\begin{align}
\Phi_i(t,0)
=
\prod_{0<t_l\leq t}
(1-q_{i,l})
\exp\bigg({-\int_{0}^{t}a_{i}(u)du}\bigg).
\label{def_phi}
\end{align}
\end{lemma}

\begin{remark}\hfill\\
The evolution operator \(\Phi_i\) satisfies the composition property
\begin{equation}
\Phi_i(t,r)\Phi_i(r,s)=\Phi_i(t,s),
\qquad s\leq r\leq t.
\label{eq:composition}
\end{equation}
Moreover, by the periodicity assumptions in {\rm(H1)}, the evolution operator satisfies
\begin{equation}
\Phi_i(t+\omega,s+\omega)=\Phi_i(t,s),
\qquad t\geq s.
\label{eq:Phi-periodic}
\end{equation}
\end{remark}

For $k\in\mathbb N_0$, every interval of the form $I_k=[t_k,t_{k+1})$ splits as
\[
I_k=I_k^+\cup I_k^-,\qquad
I_k^+=[t_k,\zeta_k),\qquad
I_k^-=[\zeta_k,t_{k+1}).
\]
This leads to the following lemma.
\begin{lemma}
\label{lem:integ_equation_2_splitting}\hfill\\
For $t\geq0$, let $k(t)\in\mathbb N_0$ be determined by $t\in I_{k(t)}$.
A function
$x(t)=x(t,0,x_0)=(x_1(t),\ldots,x_m(t))^T$
is a solution of \eqref{eq:hopf_impulsivo} on $[0,\infty)$
if and only if, for $i=1,\ldots,m$ and $t\geq 0$,
{\small
\begin{align*}
x_i(t)
={}&
\Phi_i(t,0)x_i^0
\\
&+
\sum_{r=0}^{k(t)-1}
\int_{t_r}^{t_{r+1}}
\Phi_i(t,s)
\left(
\sum_{j=1}^{m}b_{ij}(s)f_j(x_j(s))
+
\sum_{j=1}^{m}c_{ij}(s)g_j(x_j(\zeta_r))
+d_i(s)
\right)\,ds
\\
&+
\int_{t_{k(t)}}^{t}
\Phi_i(t,s)
\left(
\sum_{j=1}^{m}b_{ij}(s)f_j(x_j(s))
+
\sum_{j=1}^{m}c_{ij}(s)g_j(x_j(\zeta_{k(t)}))
+d_i(s)
\right)\,ds
\\
&+
\sum_{0<t_k\leq t}
\Phi_i(t,t_k)\widetilde h_{i,k}(x_i(t_k^-)).
\end{align*}
}
The sum over $r$ is empty when $k(t)=0$, and $\Phi_i$ is the
evolution operator defined in \eqref{def_phi}.
\end{lemma}

\section{Green function and periodic solutions for
\texorpdfstring{(\ref{eq:hopf_impulsivo})}{the Hopfield system}
under global and local conditions}
\label{sec:existstaperiodsol}

In this section, we introduce the periodic Green
function associated with the linear impulsive
system \eqref{imp_no_perturbado} and use it to
obtain an integral representation of periodic
solutions of \eqref{eq:hopf_impulsivo}.
This representation provides the basis for
the fixed-point argument used to establish
the existence and uniqueness of an
$\omega$-periodic solution.\\
For background on periodic impulsive systems,
we refer to Samoilenko and Perestyuk
\cite{SAMOILENKO}. We give a direct derivation
of the Green representation used here in
Appendix~\ref{green_appendix}.
The generalized piecewise constant argument
appears only in the nonlinear terms, so the
Green function is determined by the homogeneous
linear impulsive system.
It follows from the variation-of-constants formula
by imposing the periodic boundary condition
$x(\omega)=x(0)$ and solving for the initial value
under the noncriticality condition
$1-\Phi_i(\omega,0)\neq 0$.
Substituting this initial value into the
variation-of-constants formula yields the periodic
Green representation, with the impulsive
contributions appearing as finite sums.\\
In our setting, the generalized piecewise constant
argument $\gamma$ enters only through the nonlinear
terms and therefore does not affect the construction
of the Green function.\\

Our approach follows the ideas developed in \cite{Pinto_ejqtde}.
We refine and correct some of the estimates in that work and adapt
the arguments to the present setting, taking into account the
generalized piecewise constant argument and the periodic structure
of $\gamma(t)$.

We consider two settings. We first work under the global Lipschitz
assumptions {\rm (H3)--(H4)}. We then address the local case, in which
the Lipschitz condition {\rm (H5)} is imposed only on the ball $B[0,R]$.

 Throughout this section, for
$v=(v_1,\ldots,v_m)\in\mathbb R^m$, we use the norm
\[
\|v\|_\infty
=
\max_{1\leq i\leq m}|v_i|.
\]
For $\varphi\in PC_\omega$, the corresponding norm is
\[
\|\varphi\|_\infty
=
\sup_{0\leq t\leq\omega}\|\varphi(t)\|_\infty.
\]
\subsection{Green function}
We first establish a Poincar\'e-type criterion for the nonhomogeneous linear impulsive system \eqref{imp_no_perturbado}. 
The following lemma gives the criterion; see, for example, \cite{AKH7}.
\begin{lemma}[Poincar\'e criterion]
\label{lem:poincare}\hfill\\
Assume that {\rm (H1)} holds. Then a solution
$x(t)=x(t,0,x_0)$
of the nonhomogeneous linear impulsive system
\eqref{imp_no_perturbado} is $\omega$-periodic if and only if
\[
x(\omega)=x(0)=x_0.
\]
\end{lemma}

\begin{proof}
If $x$ is $\omega$-periodic, then clearly $x(\omega)=x(0).$

Conversely, suppose that
$x(\omega)=x(0),$
and define
\[
y(t)=x(t+\omega),
\qquad t\geq0.
\]
By {\rm (H1)}, for $t\neq t_k$,
\begin{align*}
y_i'(t)
&=x_i'(t+\omega)\\
&=-a_i(t+\omega)x_i(t+\omega)+d_i(t+\omega)\\
&=-a_i(t)y_i(t)+d_i(t).
\end{align*}
Moreover, since $t_{k+p}=t_k+\omega$, again by {\rm (H1)},
\begin{align*}
\Delta y_i\big|_{t=t_k}
&=
x_i(t_{k+p})-x_i(t_{k+p}^-)
\\
&=
-q_{i,k+p}x_i(t_{k+p}^-)+e_{i,k+p}
\\
&=
-q_{i,k}y_i(t_k^-)+e_{i,k}.
\end{align*}
Thus, $y$ satisfies the same impulsive system
\eqref{imp_no_perturbado} as $x$.

Finally,
$y(0)=x(\omega)=x(0).$
By uniqueness of solutions,
\[
y(t)=x(t),
\qquad t\geq0.
\]
Hence
\[
x(t+\omega)=x(t),
\qquad t\geq0,
\]
and therefore $x$ is $\omega$-periodic.
\end{proof}

\begin{theorem}[Green representation and unique periodic solution]\label{th:green_periodic}\hfill\\ 
Assume that conditions {\rm (H1)} and {\rm (H2)} hold.
If $x$ is an $\omega$-periodic solution of
\eqref{eq:hopf_impulsivo}, then, for $i=1,\ldots,m$ and
$0\leq t\leq\omega$,
\begin{align}
x_i(t)
={}&
\int_0^\omega
\mathcal H_i(t,s)F_i(s,x(s))\,ds
+
\sum_{k=1}^{p}
\mathcal H_i(t,t_k)
\widetilde h_{i,k}(x_i(t_k^-)),
\label{9}
\end{align}
where
\begin{align*}
F_i(t,x(t))
={}&
\sum_{j=1}^{m}b_{ij}(t)f_j(x_j(t))
+
\sum_{j=1}^{m}c_{ij}(t)g_j(x_j(\gamma(t)))
+d_i(t),
\\
\widetilde h_{i,k}(x_i(t_k^-))
={}&
I_{i,k}(x_i(t_k^-))+e_{i,k},
\end{align*}
and
\begin{equation}
\mathcal H_i(t,s)
=
\begin{cases}
\dfrac{\Phi_i(t,s)}
      {1-\Phi_i(\omega,0)},
&
0\leq s\leq t\leq\omega,
\\[5mm]

\dfrac{\Phi_i(t+\omega,s)}
      {1-\Phi_i(\omega,0)},
&
0\leq t<s\leq\omega.
\end{cases}
\label{h}
\end{equation}
The function $\mathcal H_i$ is the Green function associated with the
homogeneous linear impulsive part of \eqref{eq:hopf_impulsivo}. Moreover, if the global Lipschitz conditions $\rm{(H3)-(H4)}$ hold and $\rm{(H8)}$ is satisfied, then system \eqref{eq:hopf_impulsivo} has a unique
$\omega$-periodic solution $x^*$, and $x^*$ is the unique fixed point
of the operator defined by the right-hand side of \eqref{9}.
\end{theorem}
\begin{proof}\hfill\\
Let
\[
PC_\omega=
\left\{
\varphi\in \mathcal{PC}(\mathbb R_0^+,\mathbb R^m):
\varphi(t+\omega)=\varphi(t),\quad t\geq0
\right\}.
\]
Recall that $\mathcal{PC}(\mathbb R_0^+,\mathbb R^m)$ consists of all functions $x: \mathbb R_0^+\to\mathbb R^m$ that are continuous for $t\neq t_k$ and right-continuous, with possible jump discontinuities at the prescribed impulse times $t=t_k$, for $k\in\mathbb N$.\\
Fix $\varphi\in PC_\omega$ and consider the auxiliary impulsive system
\begin{subequations}
\label{10}
\begin{alignat}{2}
x_i'(t)&=-a_i(t)x_i(t)+F_i(t,\varphi(t)),\qquad && t\neq t_k, \label{10:eq}\\
\Delta x_i\big|_{t=t_k}
&=-q_{i,k}x_i(t_k^-)+\widetilde h_{i,k}(\varphi_i(t_k^-)),
\qquad && t=t_k,
\label{10:imp}
\end{alignat}
\end{subequations}
for $i=1,\ldots,m$.

Since $\varphi\in PC_\omega$, by {\rm (H1)} and
\eqref{eq:gamma-periodic}, we have
\[
F_i(t+\omega,\varphi(t+\omega))
=
F_i(t,\varphi(t)),
\]
and
\[
\widetilde h_{i,k+p}
\bigl(\varphi_i(t_{k+p}^-)\bigr)
=
\widetilde h_{i,k}
\bigl(\varphi_i(t_k^-)\bigr).
\]
Hence, the auxiliary system \eqref{10} has the same periodic
structure as \eqref{imp_no_perturbado}.

Applying the periodic Green representation
to the auxiliary system \eqref{10}, we obtain,
for every $\varphi\in PC_\omega$,
the unique $\omega$-periodic solution
\begin{align}
x_i(t)
={}&
\int_0^\omega
\mathcal H_i(t,s)F_i(s,\varphi(s))\,ds
+\sum_{k=1}^{p}
\mathcal H_i(t,t_k)
\widetilde h_{i,k}(\varphi_i(t_k^-)).
\label{eq:green_auxiliary}
\end{align}

Setting $\varphi=x$ recovers system \eqref{eq:hopf_impulsivo}.
Since each $\varphi\in PC_\omega$ determines a unique periodic
solution of the auxiliary system \eqref{10}, periodic solutions
of \eqref{eq:hopf_impulsivo} correspond exactly to fixed points
of the operator defined by the right-hand side of
\eqref{eq:green_auxiliary}. We therefore define
$\mathcal N:PC_\omega\to PC_\omega$ by
\begin{align}
(\mathcal N\varphi)_i(t)
={}&
\int_0^\omega
\mathcal H_i(t,s)
\Bigg(
\sum_{j=1}^{m}b_{ij}(s)f_j(\varphi_j(s))
+\sum_{j=1}^{m}c_{ij}(s)g_j(\varphi_j(\gamma(s)))
+d_i(s)
\Bigg)\,ds
\nonumber\\
&+
\sum_{k=1}^{p}
\mathcal H_i(t,t_k)
\left(
I_{i,k}(\varphi_i(t_k^-))+e_{i,k}
\right),
\qquad i=1,\ldots,m.
\label{operador}
\end{align}
The Green formulas give the periodic solutions on $[0,\omega]$.
For later times, their values are obtained by taking the time
modulo $\omega$ and evaluating the solution at the corresponding
time in $[0,\omega)$. Thus,
\[
x(t+n\omega)=x(t),
\qquad
0\leq t<\omega,\quad n\in\mathbb N_0.
\]
We extend the output of $\mathcal N$ in the same way:
\[
(\mathcal N\varphi)(t+n\omega)
=
(\mathcal N\varphi)(t),
\qquad
0\leq t<\omega,\quad n\in\mathbb N_0.
\]
The values at $t=0$ and $t=\omega$ agree by the preceding
construction.\\
For every $\varphi\in PC_\omega$, $\mathcal N\varphi$ is the
unique $\omega$-periodic solution of the auxiliary system
\eqref{10}. Therefore,
\[
\mathcal N(PC_\omega)\subset PC_\omega.
\]

Set
\[
\mathcal H:=
\max_{1\leq i\leq m}
\sup_{(t,s)\in[0,\omega]^2}
|\mathcal H_i(t,s)|.
\]

We now show that $\mathcal N$ is a contraction.
Let $\varphi,\psi\in PC_\omega$.
By the Lipschitz assumptions {\rm (H3)--(H4)} and using the definitions of
$\widetilde b$, $\widetilde c$, and $\widetilde l_k$, we obtain
\begin{align*}
\|(\mathcal N\varphi)(t)-(\mathcal N\psi)(t)\|_\infty
\leq
\mathcal H
\left(
\int_0^\omega
\bigl(\widetilde b(s)+\widetilde c(s)\bigr)\,ds
+\sum_{k=1}^{p}\widetilde l_k
\right)
\|\varphi-\psi\|_\infty.
\end{align*}

Therefore,
\begin{align*}
\|\mathcal N\varphi-\mathcal N\psi\|_\infty
\leq
\mathcal H
\left[
\omega\mathcal M(\widetilde b+\widetilde c)
+p\mathcal M(\widetilde l)
\right]
\|\varphi-\psi\|_\infty.
\end{align*}
That is,
\begin{align}
\|\mathcal N\varphi-\mathcal N\psi\|_\infty
\leq
\kappa
\|\varphi-\psi\|_\infty.
\label{17}
\end{align}

Hence, by {\rm (H8)}, $\mathcal N$ is a contraction on $PC_\omega$.

Since $PC_\omega$, endowed with the norm
\[
\|\varphi\|_\infty
=
\sup_{0\leq t\leq\omega}\|\varphi(t)\|_\infty,
\]
is a Banach space, the Banach fixed-point theorem implies that
$\mathcal N$ has a unique fixed point $x^*\in PC_\omega$; that is,
\[
\mathcal N x^*=x^*.
\]

By the construction of $\mathcal N$, $\mathcal N x^*$ is the
$\omega$-periodic solution of the auxiliary system \eqref{10}
corresponding to $\varphi=x^*$.
Since $\mathcal N x^*=x^*$, the auxiliary system \eqref{10} with
$\varphi=x^*$ coincides with \eqref{eq:hopf_impulsivo}.
Therefore, $x^*$ is an $\omega$-periodic solution of
\eqref{eq:hopf_impulsivo}.

Finally, if $x$ is any $\omega$-periodic solution of
\eqref{eq:hopf_impulsivo}, then taking $\varphi=x$ in
\eqref{eq:green_auxiliary} yields
\[
x=\mathcal N x.
\]
Thus, $x$ is a fixed point of $\mathcal N$. Since $\mathcal N$ has a
unique fixed point, $x=x^*$. Hence, $x^*$ is the unique
$\omega$-periodic solution of \eqref{eq:hopf_impulsivo}, with the Green representation \eqref{9}.
\end{proof}
\begin{remark}
The Green function $\mathcal H_i$ does not require a separate treatment
of the advanced and retarded intervals
$I_k^{+}=[t_k,\zeta_k)$ and
$I_k^{-}=[\zeta_k,t_{k+1})$.
Indeed, $\mathcal H_i$ is determined by the homogeneous linear impulsive
part of \eqref{eq:hopf_impulsivo}, whereas the generalized piecewise
constant argument $\gamma(t)$ occurs only in the nonlinear perturbation
through $F_i$.
\end{remark}
\subsection{Local Lipschitz condition}

We now consider the case in which the Lipschitz assumptions are required
only on a bounded subset of $PC_{\omega}$. Let $R>0$ and define
\[
B[0,R]
=
\left\{
\phi\in PC_{\omega}\,:\,
\|\phi\|_\infty\leq R
\right\}.
\]
Suppose that the local Lipschitz condition {\rm (H5)} holds with this choice of $R$. We also assume that {\rm (H1)--(H2)} are satisfied.

Let $\varphi$ denote the unique $\omega$-periodic solution of the
linear system \eqref{imp_no_perturbado}, whose components are given by
\begin{equation}
\varphi_i(t)
=
\int_0^\omega\mathcal H_i(t,s)d_i(s)\,ds
+
\sum_{k=1}^{p}\mathcal H_i(t,t_k)e_{i,k},
\qquad i=1,\ldots,m.
\label{eq:pho0}
\end{equation}
Write
\[
\varphi
=
(\varphi_1,\varphi_2,\ldots,\varphi_m)^T.
\]
By {\rm (H5)}, $f_j(0)=g_j(0)=I_{i,k}(0)=0$.
Evaluating \eqref{operador} at the zero function gives
\[
(\mathcal N0)_i(t)
=
\int_0^\omega\mathcal H_i(t,s)d_i(s)\,ds
+\sum_{k=1}^{p}\mathcal H_i(t,t_k)e_{i,k}
=\varphi_i(t).
\]
Hence, $\mathcal N(0)=\varphi$.

Assume that $\rm{(H8)}$ holds and that
\begin{equation}
\|\varphi\|_\infty<(1-\kappa)R,
\label{18}
\end{equation}
where $\varphi$ is defined by \eqref{eq:pho0}. Set $r=R-\|\varphi\|_\infty,$ so that \eqref{18} gives $\kappa R<r$. In particular, $r>0$.

Consider the closed ball
\[
B[\varphi,r]
=
\left\{
\phi\in PC_{\omega}\,:\,
\|\phi-\varphi\|_\infty\leq r
\right\}.
\]
For every $\phi\in B[\varphi,r]$,
\[
\|\phi\|_\infty
\leq
\|\phi-\varphi\|_\infty+\|\varphi\|_\infty
\leq
r+\|\varphi\|_\infty
=R.
\]
Hence,
\[
B[\varphi,r]\subset B[0,R].
\]

Let $\phi,\psi\in B[0,R]$.
Applying the local Lipschitz estimates in {\rm (H5)}, we obtain
\begin{align*}
\|\mathcal N\phi-\mathcal N\psi\|_\infty
&\leq
\mathcal H\sum_{i=1}^{m}
\Bigg\{
\int_0^\omega
\Bigg(
\sum_{j=1}^{m}
|b_{ij}(s)|L_j|\phi_j(s)-\psi_j(s)|
\\
&\qquad+
\sum_{j=1}^{m}
|c_{ij}(s)|\overline L_j
|\phi_j(\gamma(s))-\psi_j(\gamma(s))|
\Bigg)\,ds
\\
&\qquad+
\sum_{k=1}^{p}
l_{i,k}|\phi_i(t_k^-)-\psi_i(t_k^-)|
\Bigg\}
\\
&\leq
\mathcal H
\left(
\int_0^\omega
\bigl(\widetilde b(s)+\widetilde c(s)\bigr)\,ds
+
\sum_{k=1}^{p}\widetilde l_k
\right)
\|\phi-\psi\|_\infty
\\
&=
\mathcal H
\left(
\omega\mathcal M(\widetilde b+\widetilde c)
+
p\mathcal M(\widetilde l)
\right)
\|\phi-\psi\|_\infty
\\
&=
\kappa\|\phi-\psi\|_\infty.
\end{align*}
Since $\kappa<1$, the restriction of $\mathcal N$ to
$B[\varphi,r]$ is a contraction.

We now verify that $\mathcal N$ maps $B[\varphi,r]$ into itself.
Let $\phi\in B[\varphi,r]$. Subtracting \eqref{eq:pho0} from
\eqref{operador}, we obtain
\begin{align*}
(\mathcal N\phi)_i(t)-\varphi_i(t)
={}&
\int_0^\omega
\mathcal H_i(t,s)
\Bigg(
\sum_{j=1}^{m}b_{ij}(s)f_j(\phi_j(s))
+
\sum_{j=1}^{m}c_{ij}(s)g_j(\phi_j(\gamma(s)))
\Bigg)\,ds
\\
&+
\sum_{k=1}^{p}
\mathcal H_i(t,t_k)I_{i,k}(\phi_i(t_k^-)).
\end{align*}
Now, by {\rm (H5)}, we find
\begin{align*}
\|\mathcal N\phi-\varphi\|_\infty
&\leq
\mathcal H\sum_{i=1}^{m}
\Bigg\{
\int_0^\omega
\Bigg(
\sum_{j=1}^{m}|b_{ij}(s)|L_j|\phi_j(s)|
\\
&\qquad+
\sum_{j=1}^{m}|c_{ij}(s)|\overline L_j
|\phi_j(\gamma(s))|
\Bigg)\,ds
+
\sum_{k=1}^{p}l_{i,k}|\phi_i(t_k^-)|
\Bigg\}
\\
&\leq
\mathcal H
\left(
\int_0^\omega
\bigl(\widetilde b(s)+\widetilde c(s)\bigr)\,ds
+
\sum_{k=1}^{p}\widetilde l_k
\right)
\|\phi\|_\infty
\\
&=
\kappa\|\phi\|_\infty
\\
&\leq
\kappa R
<r.
\end{align*}
The last inequality follows from \eqref{18}, since
$r=R-\|\varphi\|_\infty$.
Therefore,
\[
\mathcal N\bigl(B[\varphi,r]\bigr)\subset B[\varphi,r].
\]

The ball $B[\varphi,r]$ is closed in the Banach space
$PC_{\omega}$ and is therefore complete.
The Banach fixed-point theorem gives a unique fixed point
$x^\ast\in B[\varphi,r]$.
By the Green representation, this fixed point is an
$\omega$-periodic solution of \eqref{eq:hopf_impulsivo}.
Conversely, every $\omega$-periodic solution in $B[\varphi,r]$
is a fixed point of $\mathcal N$, so the solution is unique
within this ball.

This proves the following local version of
Theorem~\ref{th:green_periodic}.

\begin{theorem}
\label{th:teorema3}
Assume that conditions {\rm (H1)--(H2)}, {\rm (H5)} and {\rm (H8)} hold,
and that the function $\varphi$ defined by \eqref{eq:pho0}
satisfies \eqref{18}.
Then system \eqref{eq:hopf_impulsivo} has a unique
$\omega$-periodic solution $x^\ast$ in $B[\varphi,r]$, where
\[
r=R-\|\varphi\|_\infty.
\]
\end{theorem}
\section{Exponential attraction}
\label{sec:expstab}
When artificial neural networks are used to solve optimization problems, an important question is whether the network has a unique globally asymptotically stable equilibrium \cite{Gopalsamy2004}. For networks subject to periodic variations, a similar question is whether all solutions approach the same periodic response. Having established the existence and uniqueness of the periodic solution, we now show that every other solution approaches it exponentially.

Let $x(t)=x(t,0,x_0)$ be any solution of the
normalized initial-value problem, and let $x^\ast$
be the corresponding $\omega$-periodic solution.
Both functions satisfy \eqref{eq:hopf_impulsivo}
for $t\geq 0$.

We now study the convergence of solutions toward the periodic solution \(x^\ast\) obtained under the global conditions of Theorem \ref{th:green_periodic}. The argument follows the
stability approach used in \cite{Pinto_ejqtde}, adapted here to the
present system with generalized piecewise constant argument. Throughout this section, we use the sum norm
\[
\|v\|=\sum_{i=1}^{m}|v_i|,
\qquad v\in\mathbb R^m.
\]

To measure the deviation of an arbitrary solution from $x^{\ast}$, set
\[
w(t)=x(t)-x^{\ast}(t).
\]
Substituting this change of variables into
\eqref{eq:hopf_impulsivo}, we obtain the IDEPCAG
\begin{subequations}
\label{L}
\begin{alignat}{2}
w_{i}^{\prime}(t)
&=
-a_{i}(t)w_{i}(t)
+\sum_{j=1}^{m}b_{ij}(t)\widehat{f}_{j}(x_{j}^{\ast}(t),w_{j}(t))
+\sum_{j=1}^{m}c_{ij}(t)\widehat{g}_{j}(x_{j}^{\ast}(\gamma(t)),w_{j}(\gamma(t))),
\quad && t\neq t_k,
\label{L:1}
\\
\Delta w_{i}|_{t=t_k}
&=
-q_{i,k}w_{i}(t_k^-)
+\widehat{I}_{i,k}(x_{i}^{\ast}(t_k^-),w_{i}(t_{k}^{-})),
\quad && t=t_k.
\label{L:2}
\end{alignat}
Here,
\begin{align}
\widehat{f}_{j}(x_{j}^{\ast}(t),w_{j}(t))
&=
f_{j}(w_{j}(t)+x_{j}^{\ast}(t))
-f_{j}(x_{j}^{\ast}(t)),
\nonumber\\
\widehat{g}_{j}(x_{j}^{\ast}(t),w_{j}(t))
&=
g_{j}(w_{j}(t)+x_{j}^{\ast}(t))
-g_{j}(x_{j}^{\ast}(t)),
\label{L:3}
\end{align}
and
\[
\widehat{I}_{i,k}(x_{i}^{\ast}(t_k^-),w_{i}(t_{k}^{-}))
=
I_{i,k}(w_{i}(t_{k}^{-})+x_{i}^{\ast}(t_{k}^{-}))
-I_{i,k}(x_{i}^{\ast}(t_{k}^{-})).
\]
\end{subequations}

By {\rm (H3)--(H4)}, the translated nonlinearities
$\widehat f_j$, $\widehat g_j$, and $\widehat I_{i,k}$
are Lipschitz in their second argument, with constants
$L_j$, $\bar L_j$, and $l_{i,k}$, respectively.
These constants do not depend on the first argument. Moreover,
\[
\widehat f_j(x_{j}^{\ast}(t),0)
=
\widehat g_j(x_{j}^{\ast}(t),0)
=
\widehat I_{i,k}(x_{i}^{\ast}(t_k^-),0)
=
0.
\]
Thus, the periodic solution $x^{\ast}$ of
\eqref{eq:hopf_impulsivo} is globally exponentially attractive if and only if the zero solution
of the translated system \eqref{L} has the same property. This observation allows
us to work directly with $w$.

\begin{theorem}
\label{th:teorema4}
Assume that {\rm (H1)--(H4)} and {\rm (H6)--(H9)} hold.
Then the periodic solution $x^{\ast}$ of
\eqref{eq:hopf_impulsivo} is globally exponentially attractive.
There exists a constant $C\geq1$ such that every solution
$x$ obtained from the initial-value construction satisfies
\[
\|x(t)-x^{\ast}(t)\|
\leq
C\,\|x(0)-x^{\ast}(0)\|
e^{-\widehat{\sigma}t},
\qquad t\geq 0,
\]
where
\[
\widehat{\sigma}
=
\sigma
-
\left(\mathcal M\left(
\widetilde b
+
\widetilde c
\frac{e^{\sigma\bar\theta}}{1-\vartheta}
\right)
+
\frac{p}{\omega}
\mathcal M\bigl(\ln(1+\widetilde l)\bigr)\right)
>0,
\]
and
\begin{align*}
\vartheta
&=
\sup_{k\in\mathbb N_0}
\int_{t_k}^{\zeta_k}
\left(
\widetilde b(s)
+e^{\sigma\bar\theta}\widetilde c(s)
\right)\,ds
<1.
\end{align*}
The constant $\vartheta$ is the one introduced in {\rm (H9)}.
\end{theorem}

\begin{proof}
For $t\geq 0$, the difference $w=x-x^\ast$
satisfies \eqref{L}. Consider an arbitrary solution
\[
w(t)=(w_1(t),w_2(t),\ldots,w_m(t))^T
\]
of this translated system. Using the decay estimate from
{\rm (H6)} and the Lipschitz bounds for the translated
nonlinearities, we obtain
{\small
\begin{align*}
\|w(t)\|
&=
\sum_{i=1}^{m}|w_i(t)|
\\
&\leq
e^{-\sigma t}\|w(0)\|
+
\int_{0}^t e^{-\sigma(t-s)}
\Bigg(
\sum_{i=1}^{m}\sum_{j=1}^{m}
L_j|b_{ij}(s)|\,|w_j(s)|
+
\sum_{i=1}^{m}\sum_{j=1}^{m}
\bar L_j|c_{ij}(s)|\,|w_j(\gamma(s))|
\Bigg)\,ds
\\
&\quad+
\sum_{0<t_k\leq t}
e^{-\sigma(t-t_k)}
\sum_{i=1}^{m}l_{i,k}|w_i(t_k^-)|
\\
&\leq
e^{-\sigma t}\|w(0)\|
+
\int_{0}^t e^{-\sigma(t-s)}
\left(
\widetilde b(s)\|w(s)\|
+
\widetilde c(s)\|w(\gamma(s))\|
\right)\,ds
+
\sum_{0<t_k\leq t}
e^{-\sigma(t-t_k)}
\widetilde l_k\|w(t_k^-)\|.
\end{align*}
}

Multiplying this inequality by $e^{\sigma t}$ gives
{\small
\begin{align*}
e^{\sigma t}\|w(t)\|
\leq{}&
\|w(0)\|
+
\int_{0}^t e^{\sigma s}
\left(
\widetilde b(s)\|w(s)\|
+
\widetilde c(s)e^{-\sigma\gamma(s)}
e^{\sigma\gamma(s)}\|w(\gamma(s))\|
\right)\,ds
+
\sum_{0<t_k\leq t}
e^{\sigma t_k}\widetilde l_k\|w(t_k^-)\|.
\end{align*}
}

Set
\[
z(t)=e^{\sigma t}\|w(t)\|,
\qquad
z(t_k^-)=e^{\sigma t_k}\|w(t_k^-)\|.
\]
Since $|s-\gamma(s)|\leq\bar\theta$ for $s\geq 0$,
we obtain
\begin{align}
z(t)
\leq{}&
z(0)
+
\int_{0}^t
\left(
\widetilde b(s)z(s)
+
\widetilde c(s)e^{\sigma\bar\theta}z(\gamma(s))
\right)\,ds
+
\sum_{0<t_k\leq t}
\widetilde l_k z(t_k^-)\nonumber.
\end{align}

The variation-of-constants formula can be started
at any later time $\tau\geq 0$. Repeating the
preceding estimate from such a time gives the
required inequality for every pair of times
$\tau\leq s\leq t$.

Applying the lemma yields
\begin{align}
e^{\sigma t}\|w(t)\|
&\leq
\|w(0)\|
\prod_{0<t_k\leq t}(1+\widetilde l_k)
\exp\left(
\int_{0}^t
\left(
\widetilde b(s)
+
\widetilde c(s)
\frac{e^{\sigma\bar\theta}}{1-\vartheta}
\right)\,ds
\right)
\nonumber\\
&=
\|w(0)\|
\exp\Bigg(
\int_{0}^t
\left(
\widetilde b(s)
+
\widetilde c(s)
\frac{e^{\sigma\bar\theta}}{1-\vartheta}
\right)\,ds\,+
\sum_{0<t_k\leq t}\ln(1+\widetilde l_k)
\Bigg).
\label{eq:exp-gronwall-periodic}
\end{align}

We now write
\[
t=N\omega+\delta,
\quad
0\leq\delta<\omega, 
\text{ where } \left[\dfrac{t}{\omega}\right]=N
.
\]
The function
$
\widetilde b(s)
+
\widetilde c(s)
\frac{e^{\sigma\bar\theta}}{1-\vartheta}
$
is nonnegative and $\omega$-periodic. Therefore, every
complete interval of length $\omega$ has the same integral,
and the remaining interval contributes at most one full
period. Hence,
\begin{align*}
\int_{0}^t
\left(
\widetilde b(s)
+
\widetilde c(s)
\frac{e^{\sigma\bar\theta}}{1-\vartheta}
\right)\,ds&
=
N\omega\,
\mathcal M\left(
\widetilde b
+
\widetilde c
\frac{e^{\sigma\bar\theta}}{1-\vartheta}
\right)
+
\int_{N\omega}^{t}
\left(
\widetilde b(s)
+
\widetilde c(s)
\frac{e^{\sigma\bar\theta}}{1-\vartheta}
\right)\,ds
\\
&\leq
N\omega\,
\mathcal M\left(
\widetilde b
+
\widetilde c
\frac{e^{\sigma\bar\theta}}{1-\vartheta}
\right)
+
\omega\,
\mathcal M\left(
\widetilde b
+
\widetilde c
\frac{e^{\sigma\bar\theta}}{1-\vartheta}
\right)
\\
&\leq
\mathcal M\left(
\widetilde b
+
\widetilde c
\frac{e^{\sigma\bar\theta}}{1-\vartheta}
\right)t
+
\omega\,
\mathcal M\left(
\widetilde b
+
\widetilde c
\frac{e^{\sigma\bar\theta}}{1-\vartheta}
\right).
\end{align*}

Similarly, since $\{\widetilde l_k\}$ is $p$-periodic and
$\ln(1+\widetilde l_k)\geq0$, the impulsive contribution
consists of complete blocks of $p$ terms plus at most one
incomplete block. Thus,
\begin{align*}
\sum_{0<t_k\leq t}\ln(1+\widetilde l_k)
&\leq
N\sum_{k=1}^{p}\ln(1+\widetilde l_k)
+
\sum_{k=1}^{p}\ln(1+\widetilde l_k)
\\
&=
Np\,\mathcal M\bigl(\ln(1+\widetilde l)\bigr)
+
p\,\mathcal M\bigl(\ln(1+\widetilde l)\bigr)
\\
&\leq
\frac{p}{\omega}
\mathcal M\bigl(\ln(1+\widetilde l)\bigr)t
+
p\,\mathcal M\bigl(\ln(1+\widetilde l)\bigr),
\end{align*}
where
\[
\mathcal M\bigl(\ln(1+\widetilde l)\bigr)
=
\frac1p\sum_{k=1}^{p}\ln(1+\widetilde l_k).
\]

Substituting these estimates into
\eqref{eq:exp-gronwall-periodic}, we obtain
\begin{align*}
e^{\sigma t}\|w(t)\|
\leq{}&
\|w(0)\|
\exp\Bigg(
\Bigg(
\mathcal M\left(
\widetilde b
+
\widetilde c
\frac{e^{\sigma\bar\theta}}{1-\vartheta}
\right)
+
\frac{p}{\omega}
\mathcal M\bigl(\ln(1+\widetilde l)\bigr)
\Bigg)t
+
C_0
\Bigg),
\end{align*}
where
\[
C_0
=
\omega\,
\mathcal M\left(
\widetilde b
+
\widetilde c
\frac{e^{\sigma\bar\theta}}{1-\vartheta}
\right)
+
p\,\mathcal M\bigl(\ln(1+\widetilde l)\bigr).
\]

Consequently,
\begin{align*}
\|w(t)\|
&\leq
e^{C_0}\|w(0)\|
\exp\Bigg(
-\Bigg(
\sigma
-
\left(\mathcal M\left(
\widetilde b
+
\widetilde c
\frac{e^{\sigma\bar\theta}}{1-\vartheta}
\right)
+
\frac{p}{\omega}
\mathcal M\bigl(\ln(1+\widetilde l)\bigr)\right)
\Bigg)t
\Bigg)
\\
&=
C\,\|w(0)\|
e^{-\widehat\sigma t},
\qquad t\geq 0,
\end{align*}
with $C=e^{C_0}\geq1.$ 
By {\rm (H9)}, $\widehat\sigma>0$, and therefore
\[
\|w(t)\|\longrightarrow0
\qquad\text{as}\qquad t\longrightarrow\infty.
\]
Since $w=x-x^\ast$, this proves the global exponential
attractivity of the periodic solution $x^\ast$.
\end{proof}
The estimate in Theorem \ref{th:teorema4} allows us to write the following asymptotic
formula for any solution of \eqref{eq:hopf_impulsivo}:
\begin{equation}
x(t)
=
x^{\ast}(t)
+
O\!\left(e^{-\widehat{\sigma}t}\right),
\qquad t\to+\infty.
\label{eq:asymptotic_periodic_profile}
\end{equation}
\begin{remark}
\begin{itemize}
\item[]
\item We emphasize that the constant $C_0$
accounts for the contributions of the
incomplete final period and the remaining
impulsive terms. Their inclusion improves
the corresponding estimates in
\cite{Pinto_ejqtde}.
\item The practical advantage of our approach is that the
conditions obtained can be checked without knowing
the periodic solution explicitly. Indeed, their
verification relies on estimates involving the known
system coefficients, Lipschitz constants, and impulsive
parameters, together with integrals over one period
and finite sums. Thus, despite the presence of both
advanced and delayed arguments, the criteria remain
accessible for concrete applications.
\end{itemize}
\end{remark}

\section{Examples}

\begin{example}\label{ex:hopfield_mixed}
Consider the following two-neuron impulsive Hopfield-type network with a generalized piecewise constant argument:
\begin{subequations}\label{ex:system}
\begin{align}
 x'(t)&=-2x(t)+B\tanh(x(t))+C\tanh(x(\gamma(t)))+d(t),
 &&t\geq0,\quad t\notin\mathbb N,\label{ex:system_differential}\\
 \Delta x\big|_{t=k}&=-0.1x(k^-)+0.01\tanh(x(k^-))+e,
 &&k\in\mathbb N.\label{ex:system_impulse}
\end{align}
\end{subequations}
Here $x=(x_1,x_2)^T$, and $\tanh$ acts componentwise. The coefficients, mesh, and deviating points are
\begin{equation}\label{ex:system_data}
\begin{gathered}
 B=\begin{pmatrix}0.02&0.01\\0.01&0.02\end{pmatrix},\qquad
 C=\begin{pmatrix}0.01&0.01\\0.01&0.01\end{pmatrix},\\
 d(t)=\begin{pmatrix}\sin(2\pi t)\\\cos(2\pi t)\end{pmatrix},\qquad
 e=\begin{pmatrix}0.03\\-0.03\end{pmatrix},\\
 \gamma(t)=[t]+0.4,\quad
 t_k=k,\quad \zeta_k=k+0.4,\quad I_k=[k,k+1),\\
 \omega=p=\bar\theta=1,\qquad k\in\mathbb N_0.
\end{gathered}
\end{equation}
The argument is advanced for $k\leq t<k+0.4$ and retarded for $k+0.4<t<k+1$, so it is of mixed type.

We check the assumptions of Theorems~\ref{th:green_periodic} and~\ref{th:teorema4}. Since $\tanh$ is globally Lipschitz with constant $1$, we may take
\[
L_j=\bar L_j=1,\qquad l_{i,k}=0.01,\qquad
\widetilde b=0.06,\quad \widetilde c=0.04,\quad \widetilde l_k=0.02.
\]
The periodic data give (H1)--(H4), while (H6) holds with $\sigma=2$, because $a_i(t)=2$ and $q_{i,k}=0.1$. The remaining bounds are
\begin{equation}\label{ex:impulsive_bounds}
\begin{aligned}
 \Phi_i(1,0)&=0.9e^{-2},\qquad
 \mathcal H=\frac{1}{1-0.9e^{-2}}\approx1.13869506,\\
 \rho&=0.1<1,\qquad
 \kappa=0.12\mathcal H\approx0.13664341<1,\\
 \vartheta&=0.4(0.06+0.04e^2)\approx0.14222490<1,\\
 \widehat\sigma&=2-\left(0.06+\frac{0.04e^2}{1-\vartheta}+\ln(1.02)\right)
 \approx1.57562892>0.
\end{aligned}
\end{equation}
Thus (H7)--(H9) hold, and the system has a unique globally exponentially attractive $1$-periodic solution $x^*$. From the variation-of-constants formula and $x^*(1)=x^*(0)$, we obtain
\begin{equation}\label{eq:example_periodic_initial}
\begin{aligned}
 x^*(0)={}&\frac{1}{1-0.9e^{-2}}\Bigg(0y .9\int_0^1 e^{-2(1-s)}\Bigl(
 B\tanh(x^*(s))\\
 &\qquad+C\tanh(x^*(0.4))+d(s)\Bigr)\, y ds
 +0.01\tanh(x^*(1^-))+e\Bigg)\\
 &\approx\begin{pmatrix}-0.09487953\\0.00746911\end{pmatrix}.
\end{aligned}
\end{equation}
\end{example}

\begin{figure}[H]
\centering
\captionsetup[subfigure]{font=small,skip=3pt}
\begin{subfigure}[t]{0.495\linewidth}
\centering
\includegraphics[width=\linewidth]{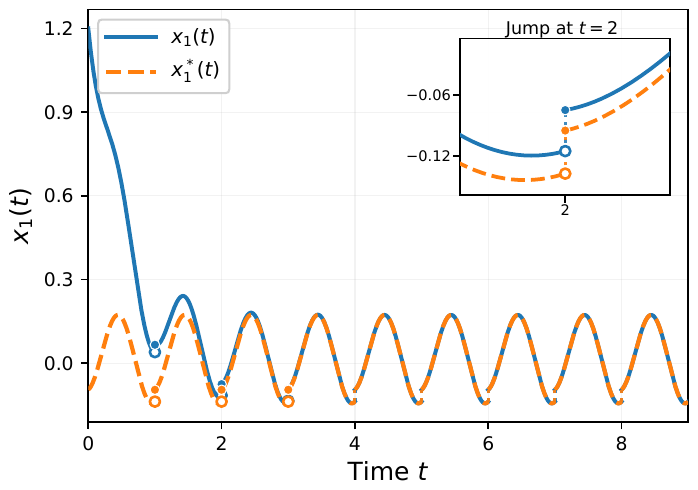}
\caption{First component.}
\end{subfigure}\hfill
\begin{subfigure}[t]{0.495\linewidth}
\centering
\includegraphics[width=\linewidth]{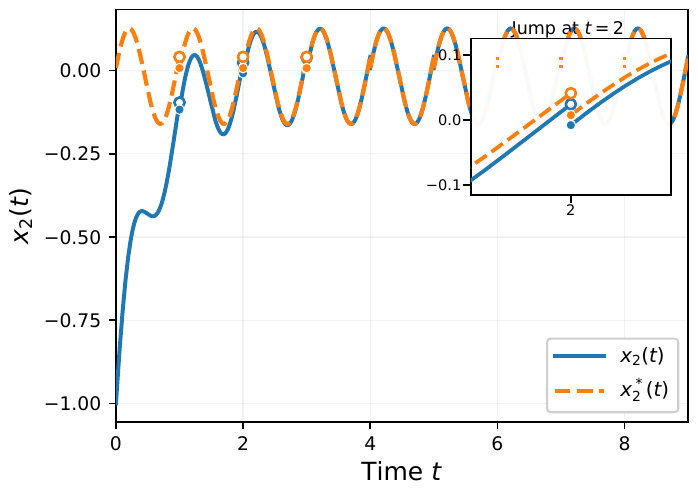}
\caption{Second component.}
\end{subfigure}
\par\vspace{0.9em}
\begin{subfigure}[t]{0.94\linewidth}
\centering
\includegraphics[width=0.7\linewidth]{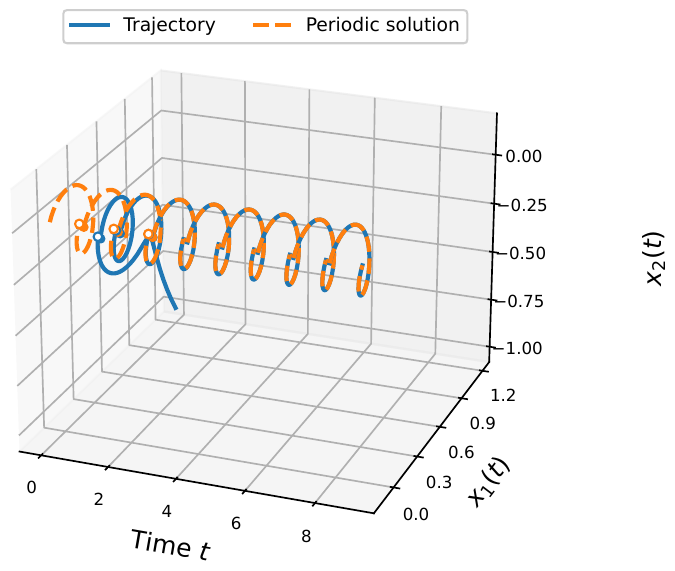}
\caption{Three-dimensional view.}
\end{subfigure}
\caption{Exponential attraction to the periodic solution from
$x(0)=(1.2,-1.0)^T$. The insets in (a) and (b) highlight the impulse
jumps at $t=2$.}
\label{fig:ex1_comparison}
\end{figure}
\newpage

\begin{example}\label{ex:nonautonomous_arctan}
We now consider the nonimpulsive, nonautonomous two-neuron DEPCAG system
\begin{equation}\label{ex:na_network}
 x'(t)=-A(t)x(t)+B(t)\arctan(x(t))+C(t)\arctan(x(\gamma(t)))+d(t),\qquad t\geq1,
\end{equation}
where $\arctan$ acts componentwise. Its coefficients, mesh, and deviating points are
\begin{equation}\label{ex:na_data}
\begin{gathered}
 A(t)=\begin{pmatrix}3+0.4\sin(\pi t)&0\\0&3.2+0.3\cos(\pi t)\end{pmatrix},\\
 B(t)=(1+0.2\cos(\pi t))\begin{pmatrix}0.02&0.01\\0.01&0.02\end{pmatrix},\\
 C(t)=(1+0.2\sin(\pi t))\begin{pmatrix}0.005&0.005\\0.005&0.005\end{pmatrix},\quad
 d(t)=\begin{pmatrix}\sin(\pi t)\\\cos(\pi t)\end{pmatrix},\\
 \gamma(t)=2\left[\frac{t+1}{2}\right],\quad
 t_k=2k+1,\quad \zeta_k=2k+2,\quad I_k=[2k+1,2k+3),\\
 \omega=\bar\theta=2,\quad p=1,\quad
 q_{i,k}=e_{i,k}=0,\quad I_{i,k}\equiv0,\qquad k\in\mathbb N_0.
\end{gathered}
\end{equation}
The argument is advanced before $\zeta_k$ and retarded after it. Here $t_0=1$; the translation $u=t-1$ brings the first mesh point to $0$ and preserves the periodicity assumptions.

Since $|\arctan u-\arctan v|\leq |u-v|$, we take $L_j=\bar L_j=1$ and $l_{i,k}=0$. Furthermore,
\[
 \widetilde b(t)=0.06(1+0.2\cos(\pi t)),\qquad
 \widetilde c(t)=0.02(1+0.2\sin(\pi t)),\qquad
 \widetilde l_k=0.
\]
The periodicity conditions (H1)--(H4) are satisfied. Since $a_1(t)\geq2.6$ and $a_2(t)\geq2.9$, (H6) holds with $\sigma=1.5$. Direct integration gives
\begin{equation}\label{ex:na_bounds}
\begin{aligned}
 \Phi_1(3,1)&=e^{-6},\qquad \Phi_2(3,1)=e^{-6.4},\\
 \mathcal H&=\frac{1}{1-e^{-6}}\approx1.00248491,\qquad \rho=0.16<1,\\
 \kappa&=\frac{0.16}{1-e^{-6}}\approx0.16039759<1,\\
 \vartheta&=0.06+e^3\left(0.02-\frac{0.008}{\pi}\right)\approx0.41056334<1,\\
 \widehat\sigma&=1.5-\left(0.06+\frac{0.02e^3}{1-\vartheta}\right)
 \approx0.75848362>0.
\end{aligned}
\end{equation}
Thus (H7)--(H9) also hold. By Theorems~\ref{th:green_periodic} and~\ref{th:teorema4}, the system has a unique globally exponentially attractive $2$-periodic solution. Using $x^*(3)=x^*(1)$ in the variation-of-constants formula gives
\begin{equation}\label{ex:na_initial_fixedpoint}
\begin{aligned}
 x^*(1)={}&\Phi(3,1)x^*(1)+\int_1^3\Phi(3,s)\Bigl(
 B(s)\arctan(x^*(s))+C(s)\arctan(x^*(\gamma(s)))+d(s)\Bigr) \,ds\\
 \approx{}&\begin{pmatrix}0.15507949\\-0.16489191\end{pmatrix},
\end{aligned}
\end{equation}
where $\Phi(t,s)=\operatorname{diag}\bigl(e^{-\int_s^t a_1(u)\,du},
 e^{-\int_s^t a_2(u)\,du}\bigr)$.

\end{example}
\begin{figure}[H]
\centering
\captionsetup[subfigure]{font=small,skip=3pt}
\begin{subfigure}[t]{0.495\linewidth}
\centering
\includegraphics[width=\linewidth]{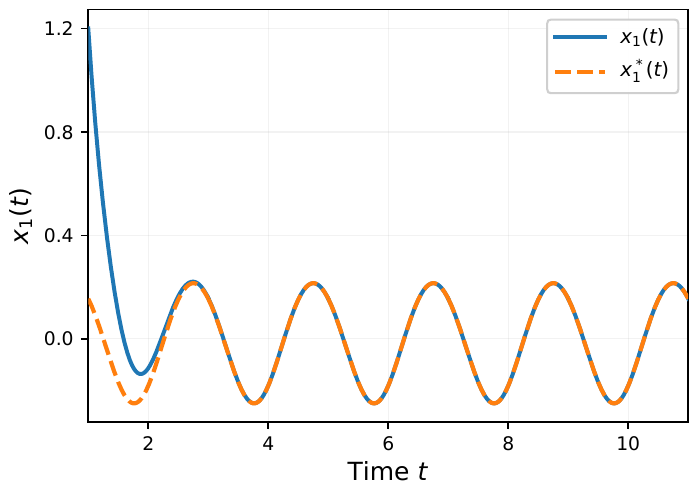}
\caption{First component.}
\end{subfigure}\hfill
\begin{subfigure}[t]{0.495\linewidth}
\centering
\includegraphics[width=\linewidth]{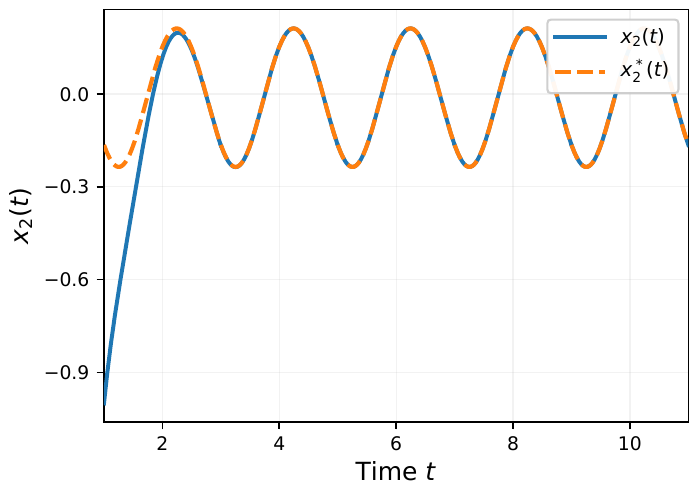}
\caption{Second component.}
\end{subfigure}
\par\vspace{0.9em}
\begin{subfigure}[t]{0.96\linewidth}
\centering
\includegraphics[width=0.7\linewidth]{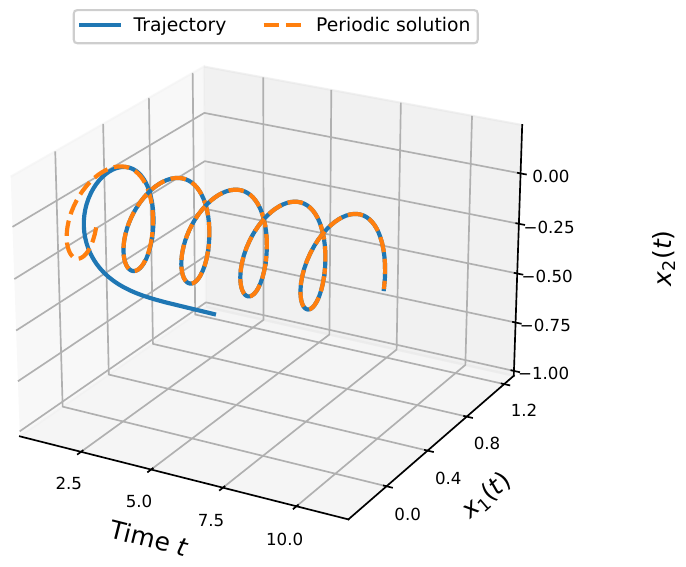}
\caption{Three-dimensional view.}
\end{subfigure}
\caption{Convergence to the $2$-periodic solution from
$x(1)=(1.2,-1.0)^T$, shown componentwise and in three dimensions.}
\label{fig:na_comparison}
\end{figure}

\section{Conclusions}
In this work, we have studied the existence,
uniqueness, and global exponential attractivity
of periodic solutions for an impulsive Hopfield-type
neural network with a generalized piecewise constant
argument, including its nonimpulsive counterpart. Using the periodic Green function associated with
\eqref{imp_no_perturbado}, we established sufficient
conditions for the existence of a unique
$\omega$-periodic solution under assumptions
{\rm (H1)--(H4)} and {\rm (H8)}.
Furthermore, under assumptions {\rm (H1)--(H4)}
and {\rm (H6)--(H9)}, we proved that this solution
is globally exponentially attractive.
An important contribution of these results is that
the long-term behavior of the system can be determined
through explicit and verifiable conditions, without
requiring an explicit expression for the periodic
solution. In particular, these conditions guarantee
not only the existence of a unique periodic response,
but also the exponential convergence of all solutions
towards it.

\begin{appendices}

\section{Derivation of the periodic Green function}
\label{green_appendix}

In this appendix, we derive the periodic Green
function associated with the linear impulsive
system \eqref{imp_no_perturbado} and obtain
the integral representation of its unique
$\omega$-periodic solution.

Although this construction follows classical
arguments, some technical details deserve
attention due to the presence of impulses
and the periodicity conditions imposed on
the coefficients and impulsive parameters.
We provide the derivation for completeness,
showing how these elements lead to the
corresponding integral formula.
Once this representation is established,
the same Green function can be used to obtain
the integral formulation of the auxiliary
system associated with \eqref{eq:hopf_impulsivo},
which provides the basis for our fixed-point
approach.

\begin{lemma}
Assume that {\rm (H1)} and {\rm (H2)} hold.
Then the nonhomogeneous linear impulsive
system \eqref{imp_no_perturbado} admits a unique
$\omega$-periodic solution, whose components
are given by
\begin{equation}
x_i(t)
=
\int_0^\omega\mathcal H_i(t,s)d_i(s)\,ds
+
\sum_{k=1}^{p}\mathcal H_i(t,t_k)e_{i,k},
\qquad i=1,\ldots,m,
\label{eq:pho0_a}
\end{equation}
where the $\omega$-periodic Green function
$\mathcal H_i(t,s)$ is the Green function
defined in \eqref{h}.
\end{lemma}
\begin{proof}
By Lemma~\ref{lem:integ_equation_2}, any solution
of \eqref{imp_no_perturbado} satisfies
\begin{align}
x_i(t)
={}&
\Phi_i(t,0)x_i^0
+\int_0^t\Phi_i(t,s)d_i(s)\,ds
+\sum_{\substack{1\leq k\leq p\\ t_k\leq t}}
\Phi_i(t,t_k) e_{i,k},
\label{11_a}
\end{align}
where $x_i^0=x_i(0)$.

If $x$ is $\omega$-periodic, then $x_i(\omega)=x_i(0)=x_i^0$.
Evaluating \eqref{11_a} at $t=\omega$, by {\rm (H2)} we have
{\small
\begin{align}
x_i^0
=
\frac{1}{1-\Phi_i(\omega,0)}
\left(
\int_0^\omega\Phi_i(\omega,s)d_i(s)\,ds
+\sum_{k=1}^{p}\Phi_i(\omega,t_k)
e_{i,k}
\right).
\label{13_a}
\end{align}
}
Conversely, let $x_i^0$ be given by
\eqref{13_a}. Evaluating \eqref{11_a}
at $t=\omega$ gives
$x_i(\omega)=x_i(0)$.
By Lemma~\ref{lem:poincare}, the resulting
solution is $\omega$-periodic.
Uniqueness follows from {\rm (H2)}, since
\eqref{13_a} uniquely determines $x_i^0$.
Substituting \eqref{13_a} into \eqref{11_a}, we obtain
{\small
\begin{align}
x_i(t)
={}&
\int_0^\omega
\frac{\Phi_i(t,0)\Phi_i(\omega,s)}
{1-\Phi_i(\omega,0)}
d_i(s)\,ds
+\int_0^t\Phi_i(t,s)d_i(s)\,ds+
\sum_{k=1}^{p}
\frac{\Phi_i(t,0)\Phi_i(\omega,t_k)}
{1-\Phi_i(\omega,0)}
e_{i,k}
+\sum_{\substack{1\leq k\leq p\\ t_k\leq t}}
\Phi_i(t,t_k)e_{i,k}.
\label{14_a}
\end{align}
}

For $s\leq t$, the composition property of the evolution operator gives
\[
\Phi_i(t,0)\Phi_i(\omega,s)
=
\Phi_i(t,s)\Phi_i(s,0)\Phi_i(\omega,s)
=
\Phi_i(t,s)\Phi_i(\omega,0),
\]
and, similarly, for $t_k\leq t$,
\[
\Phi_i(t,0)\Phi_i(\omega,t_k)
=
\Phi_i(t,t_k)\Phi_i(t_k,0)\Phi_i(\omega,t_k)
=
\Phi_i(t,t_k)\Phi_i(\omega,0).
\]
Therefore,
\begin{align*}
\Phi_i(t,s)
+\frac{\Phi_i(t,0)\Phi_i(\omega,s)}
{1-\Phi_i(\omega,0)}
&=
\frac{\Phi_i(t,s)}
{1-\Phi_i(\omega,0)},
\qquad s\leq t,\\[3mm]
\Phi_i(t,t_k)
+\frac{\Phi_i(t,0)\Phi_i(\omega,t_k)}
{1-\Phi_i(\omega,0)}
&=
\frac{\Phi_i(t,t_k)}
{1-\Phi_i(\omega,0)},
\qquad t_k\leq t.
\end{align*}

 On the other hand, if $t<s$, then the composition and periodicity
properties of $\Phi_i$ imply
\[
\Phi_i(t+\omega,s)
=
\Phi_i(t+\omega,\omega)\Phi_i(\omega,s)
=
\Phi_i(t,0)\Phi_i(\omega,s),
\]
whereas, if $t<t_k$,
\[
\Phi_i(t+\omega,t_k)
=
\Phi_i(t+\omega,\omega)\Phi_i(\omega,t_k)
=
\Phi_i(t,0)\Phi_i(\omega,t_k).
\]
Consequently,

\begin{align*}
\frac{\Phi_i(t,0)\Phi_i(\omega,s)}
{1-\Phi_i(\omega,0)}
&=
\frac{\Phi_i(t+\omega,s)}
{1-\Phi_i(\omega,0)},
\quad t<s,\qquad 
\frac{\Phi_i(t,0)\Phi_i(\omega,t_k)}
{1-\Phi_i(\omega,0)}
&=
\frac{\Phi_i(t+\omega,t_k)}
{1-\Phi_i(\omega,0)},
\quad t<t_k.
\end{align*}

 We now rewrite \eqref{14_a} in terms of the Green function
$\mathcal H_i$. Splitting the integral over $[0,\omega]$ at $t$, we obtain
{\small
\begin{align*}
&\int_0^\omega
\frac{\Phi_i(t,0)\Phi_i(\omega,s)}
{1-\Phi_i(\omega,0)}
d_i(s)\,ds
+\int_0^t\Phi_i(t,s)d_i(s)\,ds\\
&\qquad=
\int_0^t
\left(
\Phi_i(t,s)
+\frac{\Phi_i(t,0)\Phi_i(\omega,s)}
{1-\Phi_i(\omega,0)}
\right)
d_i(s)\,ds
+\int_t^\omega
\frac{\Phi_i(t,0)\Phi_i(\omega,s)}
{1-\Phi_i(\omega,0)}
d_i(s)\,ds\\
&\qquad=
\int_0^t
\frac{\Phi_i(t,s)}
{1-\Phi_i(\omega,0)}
d_i(s)\,ds
+\int_t^\omega
\frac{\Phi_i(t+\omega,s)}
{1-\Phi_i(\omega,0)}
d_i(s)\,ds\\
&\qquad=
\int_0^\omega
\mathcal H_i(t,s)d_i(s)\,ds.
\end{align*}
}

Analogously, splitting the impulsive terms according to whether
$t_k\leq t$ or $t<t_k$, we obtain
{\small
\begin{align*}
&\sum_{k=1}^{p}
\frac{\Phi_i(t,0)\Phi_i(\omega,t_k)}
{1-\Phi_i(\omega,0)}
e_{i,k}
+\sum_{\substack{1\leq k\leq p\\t_k\leq t}}
\Phi_i(t,t_k)e_{i,k}\\
&\qquad=
\sum_{\substack{1\leq k\leq p\\t_k\leq t}}
\left(
\Phi_i(t,t_k)
+\frac{\Phi_i(t,0)\Phi_i(\omega,t_k)}
{1-\Phi_i(\omega,0)}
\right)
e_{i,k}
+\sum_{\substack{1\leq k\leq p\\t<t_k}}
\frac{\Phi_i(t,0)\Phi_i(\omega,t_k)}
{1-\Phi_i(\omega,0)}
e_{i,k}\\
&\qquad=
\sum_{\substack{1\leq k\leq p\\t_k\leq t}}
\frac{\Phi_i(t,t_k)}
{1-\Phi_i(\omega,0)}
e_{i,k}
+\sum_{\substack{1\leq k\leq p\\t<t_k}}
\frac{\Phi_i(t+\omega,t_k)}
{1-\Phi_i(\omega,0)}
e_{i,k}\\
&\qquad=
\sum_{k=1}^{p}
\mathcal H_i(t,t_k)e_{i,k}.
\end{align*}
}

Thus, the unique $\omega$-periodic solution
of \eqref{imp_no_perturbado} is given by
\begin{align}
x_i(t)
={}&
\int_0^\omega
\mathcal H_i(t,s)d_i(s)\,ds
+\sum_{k=1}^{p}
\mathcal H_i(t,t_k)e_{i,k}.
\label{eq:green_auxiliary_a}
\end{align}    
\end{proof}
The same representation applies to the
auxiliary system \eqref{10} by replacing
$d_i(s)$ with $F_i(s,\varphi(s))$ and
$e_{i,k}$ with
$\widetilde h_{i,k}(\varphi_i(t_k^-))$.
This gives the integral representation
\eqref{eq:green_auxiliary}. In particular,
setting $\varphi=x$ leads to \eqref{9},
which provides the fixed-point formulation
used to establish the existence and uniqueness
of the periodic solution.
\end{appendices}


\begin{thebibliography}{99}
\bibitem{AKH3} \textsc{M. U. Akhmet}, \textit{Principles of discontinuous dynamical systems}, Springer Science \& Business Media, 2010. \doi{10.1007/978-1-4419-6581-3}

\bibitem{AKH2} \textsc{M. U. Akhmet}, \textit{Nonlinear hybrid continuous/discrete-time models}, Atlantis Studies in Mathematics for Engineering and Science,
Vol.~8, Atlantis Press, Paris, 2011. \doi{10.2991/978-94-91216-03-9}

\bibitem{AKH7} \textsc{M. U. Akhmet, C. B\"uy\"ukadal{\i}}, On periodic solutions of differential equations with piecewise constant argument, \textit{Comput. Math. Appl.} \textbf{56}(2008), No.~8, 2034--2042. \doi{10.1016/j.camwa.2008.03.031}

\bibitem{AKH8} \textsc{M. U. Akhmet, E. Y{\i}lmaz}, Impulsive Hopfield-type neural network system with piecewise constant argument, \textit{Nonlin. Anal. Real World Appl.} \textbf{11}(2010), No.~4, 2584--2593. \doi{10.1016/j.nonrwa.2009.09.003}

\bibitem{COOKE1}
\textsc{S. Busenberg, K. L. Cooke},
Models of vertically transmitted diseases with sequential-continuous dynamics,
in: \textit{Nonlinear Phenomena in Mathematical Sciences},
Academic Press, 1982, pp.~179--187.
\doi{10.1016/B978-0-12-434170-8.50028-5}

\bibitem{CTP2019}
\textsc{S.~Castillo, M.~Pinto and R.~Torres},
Asymptotic formulae for solutions to impulsive differential equations with piecewise constant argument of generalized type,
\textit{Electron. J. Differential Equations}, \textbf{2019} (2019), No.~40, 1--22.
Available at: \url{https://ejde.math.txstate.edu/Volumes/2019/40/castillo.pdf}

\bibitem{Chiu_2022}
\textsc{K.-S. Chiu},
Analyzing stability of equilibrium points in impulsive neural
network models involving generalized piecewise alternately
advanced and retarded argument,
\textit{Tamkang Journal of Mathematics}
\textbf{54}(2023), No.~4, 379--404.
\doi{10.5556/j.tkjm.54.2023.4740}

\bibitem{CHIU2026MMA}
\textsc{K.-S. Chiu},
New existence and exponential stability results for periodic solutions in
recurrent neural networks with generalized piecewise constant delay via
coincidence degree theory,
\emph{Math. Model. Anal.} \textbf{31} (2026), no.~3, 476--498.
\doi{10.3846/mma.2026.25296}

\bibitem{CHUAYANG1988}
\textsc{L. O. Chua, L. Yang},
Cellular neural networks: theory,
\textit{IEEE Trans. Circuits Syst.}
\textbf{35}(1988), No.~10, 1257--1272.
\doi{10.1109/31.7600}

\bibitem{DAI}\textsc{L. Dai}, \textit{Nonlinear dynamics of piecewise constant systems and implementation of piecewise constant arguments}, World Scientific, 2008. \doi{10.1142/6882} 

\bibitem{Gopalsamy2004} \textsc{K. Gopalsamy,} Stability of artificial neural networks with impulses, \textit{Applied Mathematics and Computation},
\textbf{154}(2004), No.~3, 783--813.
\doi{10.1016/S0096-3003(03)00750-1}

\bibitem{BOZKURT2}\textsc{F. Gurcan, F. Bozkurt}, Global stability in a population model with piecewise constant arguments, \textit{J. Math. Anal. Appl.} \textbf{360}(2009), No.~1, 334--342. \doi{10.1016/j.jmaa.2009.06.058} 


\bibitem{HOPFIELD}\textsc{J. J. Hopfield}, Neural networks and physical systems with emergent collective computational abilities,  \textit{Proc. Nat. Acad. Sci. U.S.A.} \textbf{79}(1982), No.~8,  2554--2558. \doi{10.1073/pnas.79.8.2554} 

\bibitem{HUANG2}\textsc{Z. Huang, X. Wang, F. Gao}, The existence and global attractivity of almost periodic sequence solution of discrete-time neural networks, \textit{Phys. Lett. A} \textbf{350}(2006), No.~3--4, 182--191.\doi{10.1016/j.physleta.2005.10.022} 

\bibitem{HUANG1}\textsc{Z. Huang, Y. Xia, X. Wang}, The existence and exponential attractivity of $\kappa$-almost periodic sequence solution of discrete time neural networks, \textit{Nonlinear Dynam.} \textbf{50}(2007), No.~1--2, 13--26. \doi{10.1007/s11071-006-9139-4} 


\bibitem{KARTAL}\textsc{S. Kartal}, Mathematical modeling and analysis of tumor-immune system interaction by using Lotka--Volterra predator--prey like model with piecewise constant arguments, \textit{Periodicals of Engineering and Natural Sciences} (PEN)
\textbf{2}(2014), No.~1, 7--12. \doi{10.21533/pen.v2i1.36} 

\bibitem{BEREK6}  \textsc{M. Lafci, H. Bereketoglu}, On a certain impulsive differential system with piecewise constant arguments, \textit{Math. Sci. (Springer)} \textbf{8}(2014), Art. 121, 4 pp. \doi{10.1007/s40096-014-0121-x}


\bibitem{MYSH1}\textsc{A. D. Myshkis}, On certain problems in the theory of differential equations with deviating argument, \textit{Russian Math. Surveys} \textbf{32}(1977), No.~2, 181--213. \doi{10.1070/RM1977v032n02ABEH001623}


\bibitem{BEREK1} \textsc{G. S. Oztepe, H. Bereketoglu}, Convergence in an impulsive advanced differential equations with piecewise constant argument, \textit{Bull. Math. Anal. Appl.} \textbf{4}(2012), No.~3, 57--70. \MR{2989710} 


\bibitem{OztepeKarakoç2026}\textsc{G. S. Oztepe, F. Karako\c{c}}, Discrete-Time Stability Analysis of Neural Networks with Piecewise Constant Arguments, \textit{Mathematics} \textbf{14}(2026), No.~13, 2406. \doi{10.3390/math14132406}

\bibitem{PINTO4}\textsc{M. Pinto}, Cauchy and Green matrices type and stability in alternately advanced and delayed differential systems, \textit{J. Difference Equ. Appl.} \textbf{17}(2011), No.~2, 235--254. \doi{10.1080/10236198.2010.549003}


\bibitem{PINTO6}\textsc{M. Pinto, G. Robledo}, Controllability and observability for a linear time varying system with piecewise constant delay, \textit{Acta Appl. Math.} \textbf{136}(2015), No.~1, 193--216. \doi{10.1007/s10440-014-9954-4}

\bibitem{Pinto_ejqtde} \textsc{M. Pinto, D. Sep\'ulveda and R. Torres}, Exponential periodic attractor of impulsive Hopfield-type neural network system with piecewise constant argument, \textit{Electron. J. Qual. Theory Differ. Equ.} \textbf{2018}, No. 34, 1--28. \doi{10.14232/ejqtde.2018.1.34}

\bibitem{SAMOILENKO}\textsc{A. M. Samoilenko, N. A. Perestyuk}, \textit{Impulsive differential equations}, World Scientific Series on Nonlinear Science. Series A: Monographs and Treatises, Vol. 14, World Scientific, Singapore, 1995. \doi{10.1142/2892}

\bibitem{SHAH} \textsc{S. M. Shah, J. Wiener}, Advanced differential equations with piecewise constant argument deviations, \textit{Internat. J. Math. Math. Sci.} \textbf{6}(1983), No.~4, 671--703. \doi{10.1155/S0161171283000599} 

\bibitem{WIENER1}\textsc{J. Wiener}, \textit{Generalized solutions of functional differential equations}, World Scientific, 1993. \doi{10.1142/1860}

\end{thebibliography}
\end{document}